\documentclass[a4paper, 10pt, reqno]{amsart}
\usepackage{graphicx, color} 
\usepackage{amsmath, amsthm, amsfonts, amssymb, amscd}
\usepackage[T1]{fontenc}
\usepackage{comment}
\usepackage{amsmath}
\usepackage{tikz}
\usetikzlibrary{calc}
\usepackage{latexsym,amsfonts,amssymb,amsthm,euscript, subfigure}\usepackage{amsmath}
\usepackage{tikz}
\usetikzlibrary{calc}
\usepackage{latexsym,amsfonts,amssymb,amsthm,euscript, subfigure}
\usepackage{graphicx}
\usepackage{hyperref}
\usepackage{cleveref}
\usepackage{float}
\usepackage{color,fancybox}
\usepackage[normalem]{ulem}  
\usepackage{latexsym}
\usepackage{enumitem}
\usepackage{comment}
\usepackage[british]{babel}
\usepackage{indentfirst}

\usepackage{mathrsfs}

\theoremstyle{plain}
\newtheorem{theorem}{Theorem}[section]

\newtheorem*{thB}{Theorem B}

\newtheorem{corollary}{Corollary}[section]
\newtheorem{definition}{Definition}[section]
\newtheorem{example}{Example}[section]

\newtheorem{lemma}{Lemma}[section]

\newtheorem{proposition}{Proposition}[section]
\newtheorem{remark}{Remark}[section]

\newtheorem{maintheorem}{Theorem}

\begin{document}

\title[Uniqueness and stability of equilibrium states]{Uniqueness and stability of equilibrium states for random non-uniformly expanding maps.}

\author{R. Bilbao}
\address{Rafael A. Bilbao\\ Escuela de Matem\'atica y Estat\'istica, UPTC \\ Sede Central del Norte Av. Central del Norte 39 - 115, cod. \\ 150003 Tunja, Boyac\'a \\ Colombia} 
	\email{rafael.alvarez@uptc.edu.co}

\author{V. Ramos }
\address{Vanessa Ramos \\ Departamento de Matem\'atica, UFMA \\ Av. dos Portugueses, 1966\\  65080-805 S\~ao Lu\'{i}s, Maranh\~ao \\ Brazil}
\email{ramos.vanessa@ufma.br}


\pagenumbering{arabic}

\begin{abstract}
	We consider a robust class of random non-uniformly expanding local homeomorphisms and H\"older continuous potentials with small variation. For each element of this class we develop the Thermodynamical Formalism and prove the existence and uniqueness of equilibrium states among non-uniformly expanding measures. Moreover, we show that these equilibrium states and the random topological pressure vary continuously in this setting.
\end{abstract}

\maketitle
\section{Introduction}	
The Thermodynamical Formalism, developed by Sinai, Ruelle and Bowen in the seventies and eighties, is a part of Ergodic Theory that came into existence through the application of techniques and results from statistical mechanics in the realm of smooth dynamics. One of its main goals is to describe the statistical behavior of a dynamical system via invariant measures, called \emph{equilibrium states}, that maximize the free energy of the system. 

In the classical setting, an \emph{equilibrium state} associated to a continuous transformation $T:M\rightarrow M$ defined on a compact metric space $M$ and a continuous potential $\phi:M\rightarrow\mathbb{R}$ is an invariant probability measure $\mu_{T, \phi}$ characterized by the following variational principle:
\begin{equation*}\label{PriVar}
P_{T}(\phi)=h_{\mu_{T, \phi}}(T)+\int{\phi} \,d\mu_{T, \phi}=\displaystyle\sup_{\mu\in\mathcal{M}_{T}(M)}\left\{h_{\mu}(T)+\int{\phi}\, d\mu\right\}
\end{equation*}
where $P_{T}(\phi)$ is the topological pressure, $h_{\mu}(T)$ denotes the entropy and the supremum is taken over all invariant probability measures.

This theory was initiated by the pioneering work of Sinai~\cite{Sinai} where it was proved  the existence and uniqueness of equilibrium states for Anosov diffeomorphisms and H\"older continuous potentials. In subsequent works Bowen~\cite{Bowen71} and Ruelle~\cite{Ruelle78} extended the results of Sinai to uniformly hyperbolic systems and H\"older continuous potentials. Since then, important contributions for this theory in the deterministic case have been given by several authors \cite{Climen}, \cite{OliveiraViana}, \cite{Sarig}, \cite{VarandasViana} among many others.  

In the context of random dynamical systems, the study of equilibrium states is still quite far	from being well understood despite some advances in the area. In a few words, a random dynamical system is a skew-product $F(w,x)=(\theta(w),f_{w}(x))$ where the randomness is modeled by an invertibe transformation $\theta$ preserving an ergodic measure $\mathbb{P}.$ We are interested in understanding the dynamics of compositions
$$f^{n}_{w}:=f_{\theta^{n-1}(w)}\circ...\circ f_{\theta(w)}\circ f_{w}.$$
As in the deterministic case, the random topological pressure of the system is the supremum of the entropy plus the integration of the potential among all invariant probability measures whose marginal is $\mathbb{P}$. We refer the reader to \cite{Arnold} for a background and treatment of this topic.

Once it was established a variational principle for random maps, a natural ask is for what kinds of random dynamical systems and potentials we can develop the theory of equilibrium states. In \cite{Kifer3} Kifer proved existence and uniqueness of equilibrium states for random uniformly expanding maps associated to H\"older continuous potentials. In \cite{Pei1} Liu extended this result for uniformly hyperbolic random systems. Later, the thermodynamical formalism was developed by Kifer \cite{Kifer4} for random expanding in average transformations and by Mayer, Skorulski and Urbanski~\cite{MSU} for distance expanding random mappings.  In the contex of random countable Markov shifts, the thermodynamic formalism was proved by Denker, Kifer and Stadlbauer in~\cite{DKS}.	The existence of equilibrium states with positive Lyapunov exponents was proved by Arbieto, Matheus and Oliveira~\cite{Oliveira2} for certain non-uniformly expanding maps and continuous potentials with low variation. In \cite{Bilbao} Bilbao and Oliveira obtained uniqueness of maximizing entropy measures in this context. Recently, Stadlbauer, Suzuki and Varandas \cite{StadlbauerSuzukiVarandas} developed the thermodynamical formalism for a wide class of random maps with non-uniform expansion and diferentiable potentials at high
temperature.

In this work we develop the thermodynamical formalism for a robust class of random non-uniformly expanding local homeomorphisms associated to H\"older continuous potentials with small variation. First, we prove the existence of an invariant measure absolutely continuous	with respect to the leading eigenmeasures of the dual transfer operators. This invariant measure is indeed an equilibrium state for the random dynamical system and it is unique in the setting of non-uniformly expanding measures. Moreover, we show that the random topological pressure is the integral of the leading eigenvalues of the transfer operators. As an application of our techniques, we extend the results obtained in \cite{Oliveira2} and \cite {Bilbao} for H\"older continuous potentials with small variation. 

Finally, we study the persistence of the equilibrium state under small perturba-tions of the system. In the context of SRB measures, the continuous dependence with respect to the dynamics was obtained by Alves and Viana~\cite{AV} for maps with non-uniform expansion. Such continuity was also proved by Baladi~\cite{Baladi} and Young~\cite{Young} for random perturbations of uniformly hyperbolic systems and by Alves and Ara\'ujo in~\cite{AlvesAraujo} for random perturbations of non-uniformly expanding maps.
More generaly, the continuity of the equilibrium state was proved by Castro and Varandas~\cite{Varandas1} for a class of non-uniformly expanding maps and potentials with small variation. This property was also obtained by Alves, Ramos and Siqueira~\cite{AlvesRamosSiqueira} for non-uniformly hyperbolic systems and hyperbolic potentials.  Here, we deal with a family of random non-uniformly expanding maps and potentials with small variation. We prove that the non-uniformly expanding equilibrium state as well as the random topological pressure vary continuously within this family.

We organize this paper as follows. In Section~\ref{statements}, we present our setting and state the main results. Basic definitions as random topological pressure and projective metrics are introduced in Section~\ref{preliminares}. In Section~\ref{medreferencia}, we recall the definition of reference measure and prove some properties that will be useful throughout the work. In Section~\ref{optransf}, we use the projective metric approach to obtain the thermodynamical formalism. The existence and uniqueness of equilibrium states among non-uniformly expanding measures are proved in Section~\ref{ee}. In Section~\ref{eq.stability}, we show the continuous dependence of these equilibrium states and the topological pressure as functions of the random dynamics and the potential. In the last section we describe some applications of our results.

	\section{Setting and main results}\label{statements}
Let $M$	be a compact and connected manifold with distance $d$ and $\Omega$ the space of local homeomorphisms defined on $M.$ Consider a Lesbesgue space $(X, \mathcal{A},\mathbb{P})$ and an invertible transformation $\theta:X\rightarrow X$ preserving $\mathbb{P}$. We call \emph{random dynamical system} (RDS) any continuous transformation $f:X\to \Omega$ given by $w\mapsto f_w\in \Omega$ such that $(w, x)\mapsto f_w(x)$ is measurable. For every $n\geq 0$ we define 
\[
f^0_w:=Id\ \ , \ \ f^{n}_{w}:=f_{\theta^{n-1}(w)}\circ...\circ f_{\theta(w)}\circ f_{w}\ \ \mbox{and}\ \ f^{-n}_{w}=(f^{n}_{w})^{-1}.
\]
The skew-product generated by the maps $f_w$ is the measurable transformation 
$$F:X\times M \to X\times M \ \ ; \ \ F(w, x)=(\theta(w),f_{w}(x)).$$
In particular, $F^n(w, x)=(\theta^n(w), f^{n}_{w}(x) )$ for every $n\in \mathbb{Z}.$
 
 Let $\mathcal{M}_{\mathbb{P}}(X \times M)$ be the space of probability measures on $X \times M$ such that the marginal is $\mathbb{P}$. Denote by $\mathcal{M}_{\mathbb{P}}(F)\subset\mathcal{M}_{\mathbb{P}}(X \times M)$ the set of $F$-invariant measures. Notice that, by Rokhlin's disintegration theorem \cite{Rokhlin1}, for every $\mu\in \mathcal{M}_{\mathbb{P}}(F)$ there exists a system of sample measures $\{\mu_{w}\}_{w\in X}$ of $\mu$ such that  $$d\mu(w, x)=d\mu_w(x)\ d\mathbb{P}(w).$$
 We say that a $F$-invariant measure $\mu$ is \emph{ergodic} if $(F, \mu)$ is ergodic. In what follows we assume that the system $(\theta, \mathbb{P})$ is ergodic. 
 
\subsection*{Hypothesis about the generating maps} For each $w\in X$ let $f_w:M \rightarrow M$ be a local homeomorphism satisfying: there exists a continuous function $L_w :M\rightarrow\mathbb{R}$ such that for every $x \in M$ we can find a neighborhood $U_x$ where $f_w: U_x \rightarrow f_w(U_x)$ is invertible and $$d(f_w^{-1}(y), f_w^{-1}(z)) \leq L_{w}(x)d(y, z), \ \ \mbox{for all}\ \ y,z \in f_w(U_x).$$
Notice that, the number of preimages $\#f_w^{-1}(x)$ is constant for all $x \in M$. We set $\deg(f_w):=\#f_w^{-1}(x)$ the degree of $f_w$ and assume $\deg(F)= \sup_w \deg(f_w) < \infty$.	

We suppose that there exists an open region $\mathcal{A}_w\subset M$ and constants $\sigma_w>1$ and $L_w\geq 1$ close enough to $1$ such that  
\begin{enumerate}
	\label{cond0}
	\item [(I)] $L_{w}(x) \leq L_w$ for every $x \in \mathcal{A}_w$ and $L_{w}(x) < \sigma_w ^{-1}$ for every $x \in \mathcal{A}_{w}^{c}=M\setminus \mathcal{A}_w$.
	\item[(II)] There exists a finite covering $\mathcal{U}_w$ of $M$, by open domains of injectivity for $f_w$, such that $\mathcal{A}_w$ can be covered by $q_w < \deg(f_w)$.
	\item [(III)] For every $\varepsilon>0$ we can find some positive integer $\tilde{n}=\tilde{n}(w,\varepsilon)$ satisfying $f_{\theta^j(w)}^{\tilde{n}}(B_{\theta^j(w)}(f^{j}_{w}(x),\varepsilon))=M$ for any $j\geq 0.$
	\end{enumerate}

 The conditions (I) and (II) mean that it is possible the existence of expanding and contracting behavior in $M$ but it is required for every point at least one preimage in the expanding region. The condition (III) means that the skew-product $F$ is \emph{topologically exact}.

Next we present the setting of potentials that will be considered. For $\alpha >0$, consider $C^\alpha(M)$ the space of H\"older continuous function $\varphi:M\to \mathbb{R}$ endowed with the seminorm
$$|\varphi|_\alpha=\sup_{x\neq y}\frac{|\varphi(x)-\varphi(y)|}{d(x,y)^\alpha}
$$
and the norm
$$\|\varphi\|_\alpha=\|\varphi\|_\infty+|\varphi|_\alpha,$$
where $\|\cdot\|_\infty$ stands for the $\sup$ norm. Denote by $\mathbb{L}^{1}_{\mathbb{P}}(X,C^\alpha(M))$ the space of all measurable functions  $\phi:X\times M\rightarrow \mathbb{R}$ such that  for all $w\in X$, the fiber potential $\phi_{w}:M\rightarrow \mathbb{R}$ defined by $\phi_{w}(x):=\phi(w,x)$ is H\"older continuous and $\lVert \phi \rVert_{1}=\int_{X}\lVert \phi_{w}\rVert_{\infty}\,d\mathbb{P}(w)<+\infty$. 
For $\phi\in \mathbb{L}^{1}_{\mathbb{P}}(X,C^\alpha(M))$ we assume the existence of some positive $\varepsilon_\phi>0$ satisfying for all $w\in X$ the following
\begin{equation*}\tag{IV}
\label{cond1}
\sup \phi_w - \inf \phi_w + \varepsilon_\phi< \log\deg f_w- \log q_w \quad \mbox {and}\quad 
 \left|e^{\phi_w} \right|_\alpha < \varepsilon_\phi e^{\inf\phi_w}.
\end{equation*}
Notice that all potentials $\phi\in \mathbb{L}^{1}_{\mathbb{P}}(X,C^\alpha(M))$ in a neighbourhood of zero satisfies the condition (\ref{cond1}). In the literature this class of potential is called \emph{small variation}.

Let $p_w:=\deg f_w-q_w.$ The choice of $\varepsilon_\phi$ and $L_w$ must satisfies for each $w\in X$ 
\begin{equation*}\label{condcone}\tag{V}\gamma_w:=e^{\varepsilon_\phi}\!\!\left[\frac{ p_w\sigma^{-\alpha}+ q_w L_w^{\alpha} (1+ (L_w-1)^\alpha )}{\deg(f_w)}\right]\! + \varepsilon_\phi L_w^\alpha  \left[ 1+ m(\mbox{\mbox{diam}}M)^{\alpha}\right]\leq\gamma<1
\end{equation*}

\subsection*{Statement of results} Consider $C^0(M)$ the space of real continuous functions $\psi:M\rightarrow \mathbb{R}$ endowed with the uniform convergence norm. Given $w\in X$ let $f_w:M\rightarrow M$ be the dynamics and $\phi_w:M\rightarrow \mathbb{R}$ be the potential on the fiber. The \emph{Ruelle-Perron-Fr\"obenius operator} or simply \emph{transfer operator} associated to $(f_w, \phi_w)$ is the linear operator $\mathcal{L}_{w}:C^0(M)\rightarrow C^0(M)$ defined by
$$\mathcal{L}_{w}(\psi)(x)=\sum_{y\in f_w^{-1}(x)}e^{\phi_w(y)}\psi(y).$$
Its dual operator $\mathcal{L}_{w}^{\ast}:[C^0(M)]^\ast\to[C^0(M)]^\ast $ acts on the space of Borelean measures as follows
$$\int \psi\, d\mathcal{L}_{w}^{\ast}(\rho_{\theta(w)})=\int\mathcal{L}_{w}(\psi)\, d\rho_{\theta(w)}.$$
	
In our first result we describe the Thermodynamical Formalism for random non-uniformly expanding maps.	

\begin{maintheorem} \label{formalismo}
	Consider $F:X\times M\to X\times M$ a random dynamical system satisfying conditions \mbox{(I), (II)} and \mbox{(III)}. For any potential $\phi:X\times M\to \mathbb{R}$ satisfying~(\ref{cond1})  the following holds:\\
		(1) There exists a unique measurable family of probabilities $\{\nu_{w}\}_{w\in X}$ such that $$\mathcal{L}_{w}^{\ast}\nu_{\theta(w)}=\lambda_{w}\nu_{w}\,\, \mbox{where}\,\, \lambda_{w}=\nu_{\theta(w)}(\mathcal{L}_{w}(1)),\,\,\mbox{almost every $w\in X$}.$$
				(2) There exist a unique measurable family of H\"older continuous function $\{h_{w}\}_{w\in X}$ bounded away from zero and infinity such that 
	$$ \mathcal{L}_{w}h_{w}=\lambda_{w} h_{\theta(w)}\,\, \mbox{and}\,\,\nu_w(h_w)=1
	\,\,\,\mbox{for almost every $w\in X$}.$$
	  	 	 (3) The probability measure $\mu:=\{\mu_w\}_{w\in X}$ where $\mu_w:=h_{w}\nu_w$ is $F$-invariant.
	
	\end{maintheorem}

We also derive that the family $\{\mu_w\}_{w\in X}$ has an exponential decay of correlations for H\"older continuous observables. 

\begin{maintheorem} \label{decaimento}
	There exists $0<\tau< 1$ such that for any $\varphi\in L^1(\mu_{\theta^{n}(w)})$ and $\psi \in C^{\alpha}(M) $ there exists a positive constant $K(\varphi, \psi) $ satisfying:
	$$\left|\int \left(\varphi \circ f^{n}_{w}\right) \psi \ d\mu_{w} - \int \varphi \ d\mu_{\theta^{n}(w)} \int \psi\ d\mu_{w}  \right| \leq K(\varphi, \psi) \tau^n$$
	for all $n\geq 1$.  
\end{maintheorem}

The weak hyperbolicity property of the generating maps allows us to prove that the $F$-invariant measure given by Theorem~\ref{formalismo} is indeed an equilibrium state for the random dynamical system. Moreover, it is unique if we consider only the measures whose pressure is located on the expanding region. We precise the setting as follows.

Suppose that there exists $c>0$ such that for $\mathbb{P}$-almost every $w\in X$ we can find $\tilde{L}_w$ close enough to $1$ and $\tilde{\sigma}_w>1$ satisfying for every $j\geq 0$ that
\begin{equation}\tag{VI}
\label{cond2}
L_{\theta^j(w)}\leq \tilde{L}_w\quad,\quad \tilde{\sigma}_w\leq\sigma_{\theta^j(w)} \quad \mbox{and}\quad \tilde{L}_w^{\rho}\tilde{\sigma}_w^{-(1-\rho)} < e^{-2c} <1,
\end{equation}
where $\rho$ is given by Lemma~\ref{lemma1}. For the potential, we assume $\int|\phi_{w}|_{\alpha} \ d\mathbb{P}(w)< +\infty$. 

We say that a subset $H$ of $X\times M$ is \emph{non-uniformly expanding} if there exists some positive constant $c>0$ such that
\begin{equation}\tag{$\star$}\label{H}
H:=\left\{(w, x)\in X\times M\, ;\, \limsup_{n\to+\infty} \frac{1}{n}\sum^{n-1}_{j=0}\log L_{\theta^{j}(w)}(f^{j}_{w}(x))^{-1}\leqslant -2c<0\right\}.
\end{equation}
A probability measure $\eta$, not necessarily invariant, is called \emph{non-uniformly expanding} with exponent $c$ if $\eta(H)=1$. 

Our next result establishes uniqueness of equilibrium states, among non-uniformly expanding measures, for radom dynamical systems $F|_\theta$ and potentials $\phi$ satisfying conditions (I)-(VI).
		
\begin{maintheorem}\label{esteq}
	There exists only one $F$-invariant non-uniformly expanding measure $\mu_{F, \phi}\in \mathcal{M}_{\mathbb{P}}(F)$ maximizing the variational principle 
	$$P_{F|_{\theta}}(\phi)=\int\log\lambda_w\ d\mathbb{P}(w)=h_{\mu_{F, \phi}}(F|\theta)+\int\!\phi\ d\mu_{F, \phi}=\sup\left\{h_{\mu}(F|\theta)+\int\phi\ d\mu\right\}$$  
	where the supremum is taken in set $\mathcal{M}_{\mathbb{P}}(F).$ Thus, $\mu_{F, \phi}$ is the unique non-uniformly expanding equilibrium state of $(F, \phi).$ 
\end{maintheorem}

Once we have proved uniqueness of equilibrium states, we are going to investigate its persistence under small perturbations of the random system and the potential. 

As defined above consider $\mathbb{L}^{1}_{\mathbb{P}}(X,C^\alpha(M))$ the space of integrable potentials and let $\mathcal{D}\subset\Omega$ be the space of $C^1$ local diffeomorphisms defined on $M$. We shall consider on $\mathcal{D}\times \mathbb{L}^{1}_{\mathbb{P}}(X,C^\alpha(M))$ the product topology.
We fix an invertible transformation $\theta: X\to X$ preserving an ergodic measure $\mathbb{P}$ and consider $\mathcal{S}$ the family of skew-products generated by maps of $\mathcal{D}$
$$F:X\times M \to X\times M \ \ ; \ \ F(w, x)=(\theta(w),f_{w}(x))$$
where $(w, x)\mapsto f_w\in \mathcal{D}$ is measurable. Now we define the family 
$$\mathcal{H}=\left\{(F, \phi)\in \mathcal{S}\times \mathbb{L}^{1}_{\mathbb{P}}(X,C^\alpha(M)) \, ;\, (F, \phi)\, \, \mbox{satisfying conditions (I)-(VI)}\right\}.$$
By Theorem~\ref{esteq}, each $(F, \phi)\in \mathcal{H}$ has only one non-uniformly expanding equilibrium state. Our last main result states the continuity in the weak star topology of such equilibria within this family, this property is called \emph{equilibrium stability}. We also prove the continuity of the random topological pressure in this setting.
\begin{maintheorem} \label{estabilidade}
	The non-uniformly expanding equilibrium state and the topological pressure vary continuously on $\mathcal{H}$.
\end{maintheorem}

We point out that we are fixing an invertible transformation $\theta:X\to X$ preserving an ergodic measure $\mathbb{P}.$ However, the proof of Theorem~\ref{estabilidade} remains true if we vary $\theta$ in the space of continuous functions.

	\section{Preliminaries}\label{preliminares}
	In this section we state some basic definitions and results about random dynamical systems that will be used throughout the text. We also recall the notion of hyperbolic times and projective metrics.	
	\subsection{Entropy and Topological Pressure}

Let $\mu\in \mathcal{M}_{\mathbb{P}}(F)$ be an $F$-invariant measure. Given a finite measurable partition $\xi$ of $M$ we set 
$$h_{\mu}(F|\theta; \xi):=\lim_{n\rightarrow +\infty} \frac{1}{n} \int_X H_{\mu_w}\left(\bigvee_{j=0}^{n-1}f_{w}^{-j}(\xi)\right) d\mathbb{P}(w)$$
where $H_{\nu}(\xi)=-\sum_{P\in \xi}\nu(P)\log\nu(P)$ for a finite partition $\xi$ and $\mu_w$ is the sample measure of $\mu.$
The \emph{entropy} of $(F|_{\theta}, \mu)$ is $$h_{\mu}(F|\theta):=\sup_{\xi}\{h_{\mu}(F|\theta;\xi)\}
$$
where the supremum is taken over all finite measurable partitions of $M$.

Denote by $\mathbb{L}^{1}_{\mathbb{P}}(X,C^0(M))$ the space of all measurable functions  $\phi:X\times M\rightarrow \mathbb{R}$ such that $\phi_{w}:M\rightarrow \mathbb{R}$ defined by $\phi_{w}(x):=\phi(w,x)$ is continuous for all $w\in X$ and $\lVert \phi \rVert_{1}=\int_{X}\lVert \phi_{w}\rVert_{\infty}\,d\mathbb{P}(w)<+\infty$. 

Fix $w\in X$. Given $\varepsilon >0$ and an integer $n\geq 1$, we say that a subset $F_n\subseteq M$ is $(w,n,\varepsilon)-$separated if for every two distinct points $y,z\in F_n$ there exists some $j\in \{0,1,...,n-1\}$ such that $d(f^{j}_{w}(y),f^{j}_{w}(z))>\varepsilon$. 

For $\phi \in  \mathbb{L}^{1}_{\mathbb{P}}(X,C^0(M))$, $\varepsilon >0$ and $n\geq 1$ we consider 
$$P_{F|\theta}(\phi)(w, n, \varepsilon)= \sup\left\{\sum_{y\in F_n}e^{ S_{n}\phi(w,y)} \ ;\ F_n\  \mbox{is a}\ (w,n,\varepsilon)-\mbox{separated set}\right\}$$
where $S_{n}\phi(w,y):=\sum^{n-1}_{j=0}\phi_{\theta^{j}(w)}( f^{j}_{w}(y))$.

The \emph{random topological pressure} of $\phi$ relative to $\theta$ is defined by	 
$$P_{F|\theta}(\phi)=\lim_{\varepsilon\rightarrow 0}\limsup_{n\rightarrow \infty} \frac{1}{n}\int_{X}\log P_{F|\theta}(\phi)(w, n, \varepsilon)\, d\mathbb{P}(w).$$
Thus it is well defined the pressure map as follows
$$
\begin{array}{cccc}
P_{F|\theta}\ : & \! \mathbb{L}^{1}_{\mathbb{P}}(X,C^0(M)) & \! \longrightarrow
& \!\mathbb{R}\cup\{\infty\} \\
& \! \phi & \! \longmapsto	 & \! P_{F|\theta}(\phi)
\end{array}
$$ 
In particular, the topological entropy of $F$ relative to $\theta$ is $h_{top}(F|_{\theta})=P_{F|\theta}(0)$.	

The topological pressure and the entropy are related by the well known Variational Principle. We refer the reader to \cite{Pei1} for a proof.      .

\begin{theorem}	\label{thprinvaria}
	If $X$ is a Lebesgue space then for any $\phi\in \mathbb{L}^{1}_{\mathbb{P}}(X,C^0(M))$ we have
	\begin{equation} 	\label{eqprinvaria}
	P_{F|\theta}(\phi)=\sup_{\mu\in \mathcal{M}_{\mathbb{P}}(F)}\biggl(h_{\mu}(F|\theta)+\int\phi\ d\mu \biggr).
	\end{equation}
\end{theorem}
When $\mathbb{P}$ is ergodic, we can consider the supremum over ergodic measures. 

We say that a probability measure $\mu\in \mathcal{M}_{\mathbb{P}}(F)$ is an \emph{equilibrium state} for $(F|_{\theta}, \phi)$ relative to $\theta$ if the supremum (\ref{eqprinvaria}) is attained by $\mu$, i.e.,
\[
P_{F|\theta}(\phi)=h_{\mu}(F|\theta)+\int\phi \ d\mu.
\]
	
	Next we define random topological pressure using dynamical balls. We take as reference the deterministic case where this approach is characteristic in dimension theory. We refer the reader to \cite{Pesin}.	
	
Fix $\varepsilon>0$ and $w\in X$. For $n\in \mathbb{N}$, $x\in M$, let $B_{w}(x,n,\varepsilon)$ be the dynamical ball
	\[
	B_{w}(x,n,\varepsilon):=\{y\in M: d(f_{w}^{j}(x),f_{w}^{j}(y) )< \varepsilon, \ \mbox{ for} \ \ 0\leq j \leq n \}.
	\]
	We denote by $G_{(N,w)}$ the collection of dynamical balls:
	\[
	G_{(N,w)}:=\{ B_{w}(x,n,\varepsilon): x\in M \ \mbox{and} \ n\ge N \}.
	\]
	
	Let $U_w$ be a finite or countable family of $G_{(N,w)}$ which covers $M$. For every $\beta\in \mathbb{R}$ and $\phi \in  \mathbb{L}^{1}_{\mathbb{P}}(X,C^0(M))$ let
	\[
	m_{\beta}(w, \phi, F|\theta, \varepsilon, N)=\inf_{U_{w}\subset G_{(N,w)}}\left\{ \sum_{B_{w}(x,n,\varepsilon)\in U_w} e^{-\beta n + S_{n}\phi(B_{w}(x,n,\delta))}  \right\}
		\]
	where $S_{n}\phi(B_{w}(x,n,\varepsilon))=\sup_{y\in B_{w}(x,n,\varepsilon)}\sum_{j=0}^{n-1}\phi_{\theta^{j}(w)}(f^{j}_{w}(y))$. As $N$ goes to infinity we define 
	\[
	m_{\beta}(w, \phi, F|\theta, \varepsilon)=\lim_{N\to \infty}m_{\beta}(w, \phi, F|\theta, \varepsilon, N).
	\]
	The existence of the limit above is guaranteed by the function $m_{\beta}(w, \phi, F|\theta, \varepsilon, N)$ to be increasing with $N$. Taking the infimum over $\beta$ we call
	\[
	P_{F|\theta}(w, \phi, \varepsilon)=\inf\{\beta: m_{\beta}(w, \phi, F|\theta, \varepsilon)=0 \}
	\]
Since $P_{F|\theta}(w, \phi, \varepsilon)$ is decreasing on $\varepsilon$ we can take the limit	$$P_{F|\theta} (w, \phi)=\lim_{\varepsilon\rightarrow 0} {P_{F|\theta} (w, \phi, \varepsilon)}.$$
Finally, the \emph{random topological pressure} of $(F|\theta, \phi)$ can be defined as
	\[
	P_{F|\theta} (\phi)= \int_X  P_{F|\theta} (w, \phi)\ d\mathbb{P}(w). 
\]

	\subsection{Hyperbolic times} In order to explore the non-uniform expansion of the set $H$ we need the notion of hyperbolic times. The reader can obtain more details of this concept in \cite{Alves3, Oliveira2}.

\begin{definition}\label{timehyp}
We say that $n\in \mathbb{N}$ is a $c$-hyperbolic time for $(w,x)\in X\times M$  if  
 \begin{equation}\label{eq.timehyp2}
  \prod^{n-1}_{j=n-k}L_{\theta^{j}(w)}(f^{j}_{w}(x))^{-1} \leqslant e^{-ck}, \quad \text{for every} \ 1\leqslant k\leqslant n. 
 \end{equation}
\end{definition}

It is a well known fact that if $\eta$ is a non-uniformly expanding measure with exponent $c$ then $\eta$-almost every point $(w,x)\in H$ has infinitely many $c$-hyperbolic times.
A proof of this result can be found in \cite{Alves3}.

\begin{lemma}
\label{lemacontra}
Given $c>0$ there exists $\delta=\delta(c)>0$ such that for $\mathbb{P}$-a.e. $w\in X$ holds that: if $n$ is a $c$-hyperbolic time of $(w,x)$ then the dynamical ball $B_{w}(x, n, \delta)$ around $x$ is mapped homeomorphically onto the ball $B_{\theta^{n}(w)}(f^{n}_{w}(x), \delta).$ Moreover, for $z\in B_{w}(x, n, \delta)$ and $f^{n}_{w}(z)\in B_{\theta^{n}(w)}(f^{n}_{w}(x), \delta)$ we have 
$$
d(f^{n-k}_{w}(z), f^{n-k}_{w}(x))\leq e^{-ck/2}d(f^{n}_{w}(z),f^{n}_{w}(x)), 
$$
for each $1\leq k\leq n$.
\end{lemma}

By the definition of hyperbolic times and the Lipschitz property of the inverse branches of $f_w$, if we replace $\log\|Df_w(\cdot)\|$ by $\log L_w(\cdot)$ then the proof of the lemma above is analogous to the proof of Lemma 5.5. in~\cite{Oliveira2}.

Let $\mathcal{B}$ be the Borel $\sigma$-algebra of $M.$ We say that $\xi$ is a  $\mu$-\emph{generating partition} if
$$\bigvee_{j=0}^{+\infty}f_{w}^{-j}(\xi)\equiv_{\mu}\mathcal{B}\ \ \mbox{for}\ \ \mathbb{P}-a.e.\ \ w\in X.$$ 

The next result states that every non-uniformly expanding measure admits a generating partition. See a proof of this in \cite{Oliveira2}.
\begin{lemma}\label{gp} Given $\eta$ a non-uniformly expanding measure with exponent $c>0$, consider $\delta=\delta(c)$ as in Lemma~\ref{lemacontra}. Then any measurable partition $\mathcal{P}$ of $M$ with diameter less than $\delta$ is an $\eta$-generating partition.
\end{lemma}

\subsection{Projective metrics} Consider $V$ a Banach space. We say that a subset $\mathcal{C}\subset V\setminus\{0\}$ is a \emph{cone} in $V$ if $\mathcal{C}\cap (-\mathcal{C}) = \{0\} $ and  $\lambda \cdot v\in \mathcal{C}$ for all $v\in \mathcal{C}$, $\lambda >0$.
Moreover, a cone $\mathcal{C}$ is  \emph{convex} if $v,w\in \mathcal{C}$ and $\lambda,\eta >0$  we have $\lambda \cdot v + \eta\cdot w\in \mathcal{C}.$
The closure of a cone $\mathcal{C}$, denoted by $\bar{C}$, is the set
\[
\bar{\mathcal{C}} := \left\{ w\in V| \,\mbox{there are}\, v\in \mathcal{C}\, \mbox{and}\, \lambda_n\to 0 \,\mbox{such that}\, (w+\lambda_nv)\in \mathcal{C}\, \mbox{for all}\, n\geq 1\right\}.
\]
We say that a cone $\mathcal{C} $ is \emph{closed} if $\bar{\mathcal{C}}= \mathcal{C} \cup \{0\}$.

Consider $\mathcal{C}$ a closed convex cone. Given $v, w \in \mathcal{C} $ define 
\begin{equation*}\label{A e B}
A(v,w)= \sup \left\{t>0 : w-tv \in \mathcal{C}  \right\} \ \mbox{and} \ B(v,w)=\inf \left\{s>0 : sv -w \in \mathcal{C}  \right\}.
\end{equation*}
where by convention $\sup \emptyset = 0$ and $\inf \emptyset = +\infty$. It is straightforward to check that $A(v,w)$ is finite, $B(v,w)$ is positive and $A(v,w)\leq B(v,w)$ for all $v, w \in \mathcal{C}$. 
We set  
$$ \Theta(v,w) = \log\left( \frac{B(v,w)}{A(v,w)} \right).$$ 
From the properties of $A$ and $B$ follows that $\Theta(v,w)$ is well-defined and takes values in $[0,+\infty]$. Notice that $\Theta(v,w)= 0$ if and only if $v= tw$ for some $t>0$. Therefore  $\Theta$ defines a pseudo-metric in the cone $\mathcal{C}$ and so, it induces a metric on a projective quotient space of $\mathcal{C}$. This metric is called  the \emph{projective metric of} $\mathcal{C}$.  

It is easy to verify that the projective metric depends monotonically on the cone: if $\mathcal{C}_1 \subset \mathcal{C}_2$ are two convex cones in $V$, then 
$\Theta_2(v,w) \leq \Theta_1(v,w)$ {for any} $v,w \in \mathcal{C}_1$, 
where $\Theta_1$ and $\Theta_2$ are the projective metrics in $\mathcal{C}_1$ and $\mathcal{C}_2$,  respectively.

In particular if $V_1,V_2$ are complete vector spaces and $L:{V}_1 \to {V}_2 $ is a linear operator such that $L(\mathcal{C}_{1}) \subset \mathcal{C}_{2}$ for $\mathcal{C}_1, \mathcal{C}_2$ convex cones in ${V}_{1}, {V}_{2}$ respectively then $\Theta_2(L(v),L(w)) \leq \Theta_1(v,w)$ {for any} $v,w \in \mathcal{C}_1$, where $\Theta_1$ and $\Theta_2$ are the projective metrics in $\mathcal{C}_1$ and $\mathcal{C}_2$, respectively. The next result states that $L$ will be a strict contraction if $L(\mathcal{C}_{1})$ has finite diameter in $\mathcal{C}_{2}$. 

\begin{theorem}\label{contcone}
	Let $\mathcal{C}_{1}$ and $\mathcal{C}_{2}$ be closed convex cones in the Banach spaces ${V}_1$ and ${V}_2$, respectively. If $L:V_1 \to V_2 $ is a linear operator such that $L(\mathcal{C}_{1}) \subset \mathcal{C}_{2}$ and  $\Delta = {\rm diam}_{\Theta_2}(L(\mathcal{C}_{1})) < \infty$, then
	$$\Theta_2 \left(L(\varphi), L(\psi) \right) \leq (1-e^{-\Delta}) \cdot \Theta_1 \left( \varphi, \psi \right) \quad \mbox{for all} \ \varphi, \psi \in \mathcal{C}_{1}.$$
\end{theorem}

In this work we will restrict our attention to cones of locally H\"older continuous observables.  We prove that, applying the last result, the transfer operator is a contraction in this setting. Next it follows some definitions.

We fix $\delta>0$ and we say that a function $\varphi:M\to \mathbb{R}$ is $(C, \alpha)$-H\"older continuous in balls of radius $\delta $ if for some constant $C>0$  follows that $$|\varphi(x)-\varphi(y)|\leq Cd(x, y)^\alpha \ \mbox{ for all} \  y \in B(x, \delta).$$
Denote by $|\varphi|_{\alpha,\delta}$ the smallest H\"older constant of $\varphi$ in balls of radius $\delta>0$.


The next lemma states that every locally H\"older continuous function defined on a compact and connected metric space is H\"older continuous.

\begin{lemma}\label{bolas} Let $M$ be a compact and connected metric space. Given $\delta > 0$ there exists $m \geqslant 1$ (depending only on $\delta$) such that the following holds: if $\varphi: M \to \mathbb{R}$ is $(C,\alpha)$-H\"older continuous in balls of radius $\delta$ then it is $(Cm, \alpha)$-H\"older continuous.
\end{lemma}
\begin{proof} The compactness allow us to cover $M$ with $N$ balls of radius $\delta$ where  $N$ depends only on $\delta$. Moreveover, since $M$ is connected, given $x, y\in M$ there are $z_{0}=x, z_{1},..., z_{N+1}=y$ satisfying $d(z_{i}, z_{i+1})\leq\delta$ and $d(z_{i}, z_{i+1})\leq d(x, y)$ for all $i=0,\cdots, N$. Since $\varphi$ is $(C, \alpha)$-H\"older continuous in balls of radius $\delta$ we have that	$$\left |\varphi(x)-\varphi(y)\right |\leq \sum_{i=0}^{N} \left | \varphi(z_{i})-\varphi(z_{i+1}) \right |\leq \sum_{i=0}^{N} C d(z_{i}, z_{i+1})^{\alpha}\leq C(N\!+\!1)d(x, y)^{\alpha}$$ which implies that $\varphi$ is $(C\cdot m, \alpha)$-H\"older continuous for $m=N\!+\!1$.
\end{proof}

Notice that the same argument used in the lemma above gives an estimate for the H\"older constant of $\varphi$ in balls of radius $(1+r)\delta$ for $0 < r \leq 1.$  
Indeed, let $r\in[0, 1]$ and $x, y\in M$ with $d(x, y)<(1+r)\delta$. Since $M$ is connected there exists $z\in M$ such that $d(x, z)=\delta$ and $d(z, y)<rd(x, z)$. Thus
\begin{eqnarray*}\label{eq do r}
	\left |\varphi(x)-\varphi(y)\right |&\leq& \left |\varphi(x)-\varphi(z)\right |+\left |\varphi(z)-\varphi(y)\right | \nonumber  \\
	&\leq & Cd(x, z)^{\alpha}+Cd(z, y)^{\alpha}\leq C(1+r^{\alpha})d(x, y)^{\alpha}.
\end{eqnarray*}
Therefore we conclude that if $\varphi: M \to \mathbb{R}$ is $(C, \alpha)$-H\"older continuous  in balls of radius $\delta$ then $\varphi$ is $(C(1+r^{\alpha}), \alpha)$-H\"older continuous in balls of radius $(1+r)\delta$ for each $0 < r \leq 1.$

For each $k>0$ we consider the convex cone of locally H\"older continuous observables defined on $M$ by 
\begin{equation} \label{cone holder}
\mathcal{C}^{k}_{\delta}= \left\{ \varphi:M\to \mathbb{R} : \varphi>0 \ \mbox{and} \  \frac{|\varphi|_{\alpha,\delta}}{\inf \varphi} \leq k \right\}.
\end{equation}


It follows by definition that $\mathcal{C}^{k_1}_{\delta} \subset \mathcal{C}^{k_2,}_{\delta}$ if $k_1 \leq k_2$.

From Lemma~\ref{bolas} and from the definition of $|\varphi|_{\alpha,\delta}$ we have that 
\begin{equation}\label{supcone}	\sup \varphi - \inf\varphi \leq |\varphi|_{\alpha,\delta}\cdot m \cdot d(x, y)^{\alpha} \leq (\inf \varphi \cdot k )\cdot m\cdot (\mbox{diam}\,M)^{\alpha}
\end{equation}
and thus $\sup \varphi \leq \inf \varphi\cdot  (1+ m(\mbox{diam}\,M)^{\alpha}k)$ for any $\varphi \in \mathcal{C}^{k}_{\delta}.$

In the cone $\mathcal{C}^{k}_{\delta}$ of locally H\"older continuous observables we can give a more explicit expression for the projective metric. We refer the reader to \cite{Varandas1} for its proof.  

\begin{lemma}\label{metrica cone}
	The projective metric $\Theta_k$ in the cone $\mathcal{C}^{k}_{\delta}$ is given by \[\Theta_k(\varphi , \psi)= \log \left( \frac{B_k(\varphi, \psi)}{ A_k(\varphi, \psi) }\right),\] where  
	$$A_k(\varphi, \psi):=\displaystyle\inf_{d(x,y)< \delta, z\in M} \frac{k|x-y|^{\alpha}\psi(z) - (\psi(x)-\psi(y))}{k|x-y|^{\alpha}\varphi(z) - (\varphi(x) - \varphi(y))} $$
	and $$B_k(\varphi, \psi) := \displaystyle\sup_{d(x,y)< \delta, z\in M} \frac{k|x-y|^{\alpha}\psi(z) - (\psi(x)-\psi(y))}{k|x-y|^{\alpha}\varphi(z) - (\varphi(x) - \varphi(y))}.$$
	In particular, we have that  $$A_k(\varphi, \psi)\leq \inf_{x\in M}\left\{\frac{\varphi(x)}{\psi(x)}\right\} \quad \mbox{and} \quad B_k(\varphi, \psi)\geq \sup_{x\in M}\left\{\frac{\varphi(x)}{\psi(x)} \right\}.$$
\end{lemma}
From the expression of the projective metric in the cone $\mathcal{C}^{k}_{\delta}$ one can prove that its diameter is finite for $k$ large enough, see~\cite{Varandas1}.

\begin{proposition}\label{diamfinito}
	For $0 <\gamma < 1$, the cone $\mathcal{C}^{\gamma k}_{\delta}$ has finite diameter in $\mathcal{C}^{k}_{\delta}$.
\end{proposition}
%
		
	\section{Reference Measure} \label{medreferencia}
For $w\in X$, let  $f_w:M\rightarrow M$ be the fiber dynamics and $\phi_w:M\rightarrow \mathbb{R}$ be the potential. Consider $\mathcal{L}_{w}:C^0(M)\rightarrow C^0(M)$ the transfer operator	associated to $(f_w, \phi_w)$ defined by
$$\mathcal{L}_{w}(\psi)(x)=\sum_{y\in f_w^{-1}(x)}e^{\phi_w(y)}\psi(y).$$
Consider also its dual operator $\mathcal{L}_{w}^{\ast}:[C^0(M)]^\ast\to[C^0(M)]^\ast $ which satisfies
$$\int \psi\, d\mathcal{L}_{w}^{\ast}(\rho_{\theta(w)})=\int\mathcal{L}_{w}(\psi)\, d\rho_{\theta(w)}.$$

We say that a probability measure $\nu_w\in \mathcal{M}^1(M)$ is a \emph{reference measure} associated to $\lambda_w\in\mathbb{R}$ if $\nu_w$ satisfies 
$$\mathcal{L}_{w}^{\ast}(\nu_{\theta(w)})=\lambda_w\nu_w.$$

As in the deterministic case, by applying the Schauder-Tychonoff fixed point theorem, it is straightforward to prove the existence of a system of reference measures $\{\nu_w\}_{w\in X}$ where $\nu_w$ is associated to $\lambda_w$ given by \begin{equation}\label{lambda}\lambda_{w}=\mathcal{L}^{\ast}_{w}\nu_{\theta(w)}(1)=\nu_{\theta(w)}(\mathcal{L}_{w}(1))
\end{equation}
for $\mathbb{P}$-almost every $w\in X.$ See \cite{Urbanski2} for details. In the sequel we derive some properties of the reference measure.

The \emph{jacobian} of a measure $\eta$ with respect to $f$ is a measurable function $J_{\eta}f$ s.t. $$\eta(f(A))=\int_{A}J_{\eta}f d\eta$$
for any measurable set $A$ where $f|_{A}$ is injective.

\begin{lemma}
	\label{lemmajacobiano}
	The jacobian of $\nu_w$ with respect to $f_w$ is given by $J_{\nu_w}f_w = \lambda_{w}e^{-\phi_w}$. Moreover $\nu_w$ is an open measure. In particular, supp$(\nu_w)=M$.
\end{lemma}

\begin{proof}
	Let $A\subset M$ be a measurable set such that $f_{w}|_{A}$ is injective.  Take a bounded sequence $\{\zeta_n \}\in C^0(M)$ such that $\zeta_{n}\to \mathcal{X}_{A}$. Then
	\begin{eqnarray*}
		\int_{M} \lambda_{w}e^{-\phi_w}\zeta_{n} d\nu_{w}&=&\int_{M} e^{-\phi_w}\zeta_{n} d(\mathcal{L}^{\ast}_{w}\nu_{\theta(w)})=\int_{M} \mathcal{L}_{w}(e^{-\phi_w}\zeta_{n})(y) d\nu_{\theta(w)}(y)\\
		&=& \int_{M} \sum_{f_{w}(z)=y}\zeta_{n}(z)d\nu_{\theta(w)}(y)=\int_{M} \sum_{f_{w}(z)=y} \zeta_{n}(f^{-1}_{w}(y))d\nu_{\theta(w)}(y).
	\end{eqnarray*}
	Since $\int_{M} \sum_{f_{w}(z)=y} \zeta_{n}(f^{-1}_{w}(y))d\nu_{\theta(w)}(y)  \longrightarrow \int_{M} \mathcal{X}_{A}(f^{-1}_{w}(y))d\nu_{\theta(w)}(y)$ when $n\to\infty$
	and $\int_{M} \mathcal{X}_{A}(f^{-1}_{w}(y))d\nu_{\theta(w)}(y)=\int_{M}\mathcal{X}_{f_{w}(A)} d\nu_{\theta(w)}=\nu_{\theta(w)}(f_{w}(A)).$
	We conclude that
	$$	\nu_{\theta(w)}(f_{w}(A))=\int_{A} \lambda_{w}e^{-\phi_w} d\nu_{w}.$$
	Notice that, by induction, we have for every $n\in\mathbb{N}$ that	
	\begin{equation}
	\label{equaintlocal}
	\nu_{\theta^{n}(w)}(f^{n}_{w}(A))=\int_{A}\lambda^{n}_{w}e^{-S_{n}\phi_{w}} \ d\nu_{w},
	\end{equation}
	where $\lambda^{n}_{w}= \lambda_{\theta^{n-1}(w)} \lambda_{\theta^{n-2}(w)}\cdot \cdot \cdot \lambda_{\theta(w)}\lambda_{w}$.
	
	Now we prove that $\nu_w$ is an open measure. By contradiction suppose the existence of some non-empty open set $U_w\subset M$ such that $\nu_{w}(U_w)=0$. By the exactness assumption, we can take  $\tilde{n}\in \mathbb{N}$ such that $f^{\tilde{n}}_{w}(U_w)=M$. Partitioning $U_w$ into mensurable subsets $U_{w,1},...,U_{w,k}$ where $f_{w}^{\tilde{n}}|_{U_{w,j}}$ is injective for $j=1,...,k$ we have
	$$\nu_{\theta^{\tilde{n}}(w)}(M)\leq\sum^{k}_{j=1}\nu_{\theta^{\tilde{n}}(w)}\bigl( f^{\tilde{n}}_{w}(U_{w,j})\bigr)=\sum^{k}_{j=1}\int_{U_w,j}J_{\nu_w}f^{\tilde{n}}_{w} d\nu_{w}=0$$
	which is a contradiction. This completes the proof.
\end{proof}

In the next proposition we show that the family $\{\nu_w\}_w$ satisfies a Gibbs property at hyperbolic times.

\begin{proposition} \label{conforme} Let $n$ be a hyperbolic time for $(w, x)$. For every $0<\varepsilon\leq\delta$ there exist $K_\varepsilon(w)>0$ and $0<\gamma_\varepsilon(\theta^n(w))\leq 1$ such that for all $y\in B_{w}(x,n, \varepsilon)$ holds
	$$\gamma_\varepsilon(\theta^n(w))K_{\varepsilon}(w)^{-1}\leq{\displaystyle{\frac{\nu_w(B_{w}(x,n, \varepsilon))}{\displaystyle{\exp\left(S_{n}\phi_w(y)-\log\lambda_{w}^{n}\right)}}}}\leq K_{\varepsilon}(w)$$
	where $S_{n}\phi_w(y)=\sum^{n-1}_{j=0}\phi_{\theta^{j}(w)}( f^{j}_{w}(y))$ and $\lambda_{w}^{n}=\lambda_w\lambda_{\theta(w)}\cdots\lambda_{\theta^{n-1}(w)}.$
\end{proposition}

\begin{proof} Fix $0<\varepsilon\leq\delta$. From Lemma~\ref{lemacontra} and condition~(\ref{cond1}) we get 
	\begin{eqnarray*}
		|S_{n}\phi_w(z)-S_{n}\phi_w(y)|&\leq& \sum_{k=0}^{n-1}|\phi_{\theta^{n-k}(w)}(f_w^{n-k}(z))-\phi_{\theta^{n-k}(w)}(f_w^{n-k}(y))|\\
		&\leq& \sum_{k=0}^{n-1} |\phi_{\theta^{n-k}(w)}|_{\alpha}e^{-ck/2}d(f_w^{n}(z), f_w^{n}(y))\\
		&\leq& \varepsilon \sum_{k=0}^{\infty} |\phi_{\theta^{k}(w)}|_{\alpha}e^{-ck/2}\leq K_\varepsilon(w)
	\end{eqnarray*}
	for every $z, y\in B_{w}(x,n, \varepsilon).$ By applying once again Lemma~\ref{lemacontra} we know that $f_w^{n}$ maps homeomorphically $B_{w}(x,n, \varepsilon)$ into the ball $B_{\theta^{n}(w)}(f_w^{n}(x),\varepsilon)$. Hence since the jacobian of $\nu_w$ is bounded away from zero and infinity we can write 
	$$0<\gamma_\varepsilon(\theta^n(w))\leq\nu_{\theta^n(w)}(f_w^{n}(B_{w}(x,n, \varepsilon)))=\int_{B_{w}(x,n, \varepsilon)}{\lambda_w^{n}e^{-S_{n}\phi_w(z)}}d\nu_w\leq1$$ 
	where $\gamma_\varepsilon(\theta^n(w))$ depends only on the radius $\varepsilon$ of the ball $B_{\theta^{n}(w)}(f_w^{n}(x),\varepsilon)$. Therefore for every $y\in B_{w}(x,n, \varepsilon)$ follows that
	\begin{eqnarray*}
		\gamma_{\varepsilon}(\theta^n(w))\leq\int_{B_{w}(x,n, \varepsilon)}\!{\lambda_w^{n}e^{-S_{n}\phi_w(z)}}d\nu_w\!\!\!\!
		&=&\!\!\!\!\displaystyle\int_{B_{w}(x,n, \varepsilon)}{\!\!\lambda_w^{n}e^{-S_{n}\phi_w(y)}\left(\frac{\lambda_w^{n}e^{-S_{n}\phi_w(z)}}{\lambda_w^{n}e^{-S_{n}\phi_w(y)}}\right)}d\nu_w\\ \\ 
		&\leq& K_{\varepsilon}(w)e^{-S_{n}\phi_w(y)+\log\lambda^n_w}\nu_w(B_{w}(x,n, \varepsilon)).
	\end{eqnarray*}
	Applying the same argument we have that
	$$e^{-S_{n}\phi_w(y)+\log\lambda^n_w}\nu_w(B_{w}(x,n, \varepsilon))\leq K_{\varepsilon}(w)\int_{B_{w}(x,n, \varepsilon)}{\!\!\lambda_w^{n}e^{-S_{n}\phi_w(y)}\left(\frac{\lambda_w^{n}e^{-S_{n}\phi_w(z)}}{\lambda_w^{n}e^{-S_{n}\phi_w(y)}}\right)}d\nu_w$$
	which completes the proof.
\end{proof}

\begin{remark} \label{K} It is possible to obtain a lower bound for $\gamma_\varepsilon(\theta^n(w))$.  Indeed, from hypothesis we may find $\tilde{n}=\tilde{n}(w, \varepsilon)$ such that
	$f_{\theta^n(w)}^{\tilde{n}}(B_{\theta^n(w)}(f_w^n(x),\varepsilon))=M$ and by definition of jacobian follows that  
	\begin{eqnarray*}
		1&=&\nu_{\theta^{n+\tilde{n}}(w)}(f_{\theta^n(w)}^{\tilde{n}}(B_{\theta^n(w)}(f_w^n(x),\varepsilon)))\\&\leq&\displaystyle \int_{B_{\theta^n(w)}(f_w^n(x),\varepsilon)} \lambda^{\tilde{n}}_{\theta^n(w)} e^{-S_{\tilde{n}}\phi_{\theta^n(w)}}\ d\nu_{\theta^{n}(w)}\\
		&\leq& \lambda^{\tilde{n}}_{\theta^n(w)}e^{-\tilde{n}\inf \phi_{\theta^n(w)}}\nu_{\theta^{n}(w)}(B_{\theta^n(w)}(f_w^n(x),\varepsilon)).
	\end{eqnarray*}
	Thus $ e^{\tilde{n}\inf\phi_{\theta^n(w)}-\log\lambda^{\tilde{n}}_{\theta^n(w)}}\leq\gamma_{\varepsilon}(\theta^n(w)).$ Since $\tilde{n}$ depends only on $w\in X$ and $\varepsilon>0$ we conclude that $\gamma_{\varepsilon}(\theta^n(w))$ is uniformly bounded.
\end{remark}

Consider $c>0$ given by condition~(\ref{cond2}). Given $w\in X$, let $H_w\subset M$ be the subset of $M$ such that $(w, x)$ has infinitely many hyperbolic times, i.e.,
$$H_w:=\left\{x\in M\ \ ;\ \ \limsup_{n\to+\infty} \frac{1}{n}\sum^{n-1}_{j=0}\log L_{\theta^{j}(w)}(f^{j}_{w}(x))^{-1}\leqslant -2c<0\right\}.$$
Next we prove that $\nu_w(H_w)=1$ for $\mathbb{P}$-almost every $w\in M.$

Recall that we fix $\varepsilon_\phi>0$ small satisfying $\varepsilon_\phi<\inf_w(\log (deg f_w) - \log q_w)$. In view of (\ref{cond1}) we may find $0<\varepsilon_0 <\varepsilon_\phi$ such that 
\begin{equation}
\label{equationvarepsilon2}
\sup \phi_w - \inf \phi_w + \varepsilon_0 < \log (deg f_w) - \log q_w  \quad \text{for all} \quad w\in X.
\end{equation}

Let $\mathcal{P}$ be a partition of $M$  with cardinality $\# \mathcal{P} = k$. We suppose without loss of generality that the set $\mathcal{A}_w$ is contained in the first $q_w$ elements of $\mathcal{P}$ for all $w\in X$. Consider the numbers $$\bar{p}_w = k - q_w, \,\,\,\hat{q} = \sup_{w \in X} q_{w},\,\,\, \bar{q}= \inf_{w\in X} q_{w}\,\,\,  \mbox{and}\,\,\, \hat{p} = \sup_{w\in X} \bar{p}_{w}.$$
This numbers are well defined since we assume that $\deg(F)= \sup_w \deg(f_w) < \infty$.

For $\rho \in (0,1)$ and $n\in\mathbb{N}$ let $I(\rho, n)$ be the set of itinerates 
$$I(\rho, n)\!=\!\{(i_w,..., i_{\theta^{n-1}(w)})\in \{1,...,k \}^{n};\# \{0 \leq j \leq n-1 : i_{\theta^{j}(w)}\leq q_{\theta^{j}(w)} \} > \rho n \}$$
and consider
$$C_\rho := \limsup_{n}\dfrac{1}{n} \log \# I(\rho, n).$$

\begin{lemma}[\cite{VarandasViana}, Lemma 3.1]
	\label{lemma1}
	Given $\varepsilon > 0$ there exists $\rho_0 \in (0,1)$ such that $C_\rho < \log \hat{q} + \varepsilon$ for every $\rho \in (\rho_0 , 1)$. 
\end{lemma}

\begin{proof}Notice that
	$ \# I(\rho, n) \leq \sum_{k = [\rho n]}^n \binom{n}{k} q_w q_{\theta(w)}\cdot \cdot \cdot q_{\theta^{k-1}(w)}p_w p_{\theta(w)}\cdot \cdot \cdot p_{\theta^{n -(k-1)}(w)}.
	$
	By applying Stirling's formula we have
	\[
	\sum_{k=[\rho n]}^{n} \binom{n}{k} = \dfrac{n}{2}\binom{n}{[\rho n]} \leq C_1 \exp{(2t(1 - \rho)n)}, \quad \mbox{for} \quad \rho > \dfrac{1}{2}.
	\]
	Thus there exist $C_1$ and $t > 0$ such that $
	\# I(\rho, n) \leq C_1 \exp{(2t(1 - \rho)n)} \ \hat{q}^n \ \hat{p}^{(1-\rho)n}.
	$
	Taking the limit when $n$ goes to infinity we have $$ C_\rho = \limsup_n \frac{1}{n}\log \# I(\rho, n) \leq \log \hat{q} + \varepsilon
	$$
	for any $\rho$ close enough to 1.
\end{proof}

From this lemma we can fix of $\rho < 1$ such that $$C_\rho < \log \hat{q} + \dfrac{\varepsilon_0}{4}.$$ 
Recalling the definition of $\lambda_w$ in ($\ref{lambda}$) and the equation (\ref{equationvarepsilon2}) we have that
\begin{align*}
\lambda_w \geq \deg f_w e^{\inf \phi_w}\geq e^{\log (\deg f_w) + \sup \phi_w - \log (\deg f_w) + \log q_w + \varepsilon_0}  = e^{(\log q_w + \sup \phi_w + \varepsilon_0)}.
\end{align*}
Now, using Lemma~\ref{lemmajacobiano} we obtain that
\begin{equation}
\label{equadejacobiano1}
J_{\nu_w}f_w = \lambda_w e^{-\phi_w} \geq e^{(\sup \phi_w +\log q_w + \varepsilon_0 - \phi_w)} \geq e^{\log q_w + \varepsilon_0} > q_w .
\end{equation}

\begin{proposition} \label{expanding} We have $\nu_w(H_w)=1$ for a.e. $w\in X$.
\end{proposition}
\begin{proof}
	Given $n\in\mathbb{N}$ denote by $B_w(n)$ the set of points $x\in M$ whose frequency of visits to $\{\mathcal{A}_{\theta^{j}(w)}\}_{0\leq j \leq n-1}$ up to time $n$ is at last $\rho$, that means, 
	\[
	B_w(n) =\left\{x\in M | \ \dfrac{1}{n} \# \{0 \leq j \leq n-1 : \ f^{j}_{w}(x) \in A_{\theta^j (w)}  \} \geq \rho  \right\}.
	\]
	Let $\mathcal{P}^{(n)}$ be the partition $\bigvee_{j=0}^{n-1}(f_w^j)^{-1}\mathcal{P}$. We cover $B_w(n)$ by elements of $\mathcal{P}^{(n)}$ and since $f^n_w$ is injective on every $P\in \mathcal{P}^{(n)}$, we may use (\ref{equadejacobiano1}) to obtain
	\begin{align*}
	1 \geq \nu_{\theta^n(w)}(f^n_w (P)) &= \int_P J_{\nu_w}(f^n_w) \ d\nu_w = \int_P \prod_{j=0}^{n-1}J_{\nu_{\theta^{j}(w)}}f_{\theta^j (w)} \ d\nu_w\\
	& \geq \prod_{j=0}^{n-1} e^{(\log q_{\theta^j(w)} + \varepsilon_0)}\nu_w (P) \geq e^{(\log \bar{q} + \varepsilon_0)n}\nu_w (P).
	\end{align*}
	Thus	$$
	\nu_w (P) \leq e^{-(\log \bar{q} + \varepsilon_0)n}.
	$$
	Since we can assume that $\hat{q}_w< \bar{q}_we^{\varepsilon_0 /2}$ for every $w\in X$ we have  
	\begin{align*}
	\nu_w (B_w(n)) & \leq \# I(\rho,n) e^{-(\log \bar{q}_w + \varepsilon_0)n} \\
	& \leq e^{(\log \hat{q}_w + \varepsilon_0/4)n}e^{-(\log \bar{q}_w + \varepsilon_0)n}
	\leq e^{(\log \hat{q}_w/\bar{q}_w - \varepsilon_0/2)n}.
	\end{align*}
	Hence the measure $\nu_w( B_w(n))$ decreases exponentially fast when $n$ goes to infinity. 
	Applying the Borel-Cantelli lemma we conclude that $\nu_w$-almost every point belongs to $B_w(n)$ for at most finitely many values of $n$. Then, in view of our choice (\ref{cond2}) we obtain for $n$ large enough that 	\[		\sum_{j=0}^{n-1}\log L_{\theta^j(w)}(f^{j}_{w}(x)) \leq \rho \log \tilde{L}_w + (1-\rho)\log \tilde{\sigma}^{-1}_w \leq -2c < 0
	\]
	which proves that $\nu_w$-almost every point has infinitely many hyperbolic times.
\end{proof}
Notice that from the last proposition and recalling that $\nu_w$ is an open measure we conclude that $H_w$ is dense in $M$.

	\section{Transfer Operator}	\label{optransf}
	Here we prove Theorem~\ref{formalismo} and Theorem~\ref{decaimento}. We use the projective metric approach to show that the transfer operator is a contraction in some cone of locally H\"older continuous functions. This contraction implies the existence of the invariant family $\{h_w\}_w$ uniformly bounded away from zero and infinity. Recalling the reference measure $\nu_w$ constructed in previous section, we define the probability measure $\mu_w:=h_w\nu_w$. From the exponential approximation of functions in the cone to the family $\{h_w\}_w$ we derive that $\mu_w$ has an exponential decay of correlations.

	\subsection{Invariant family}
For the construction of the invariant family $\{h_w\}_w$ we follow the ideas of Castro and Varandas~\cite{Varandas1}.

Recall that we fix $\delta>0$ and consider for each $k>0$  the cone of locally H\"older continuous functions 
\begin{equation} \label{cone holder}
\mathcal{C}^{k}_{\delta}(w)= \left\{ \varphi_{w}:M\to \mathbb{R} : \varphi_{w}>0 \ \mbox{and} \  \frac{|\varphi_{w}|_{\alpha,\delta}}{\inf \varphi_{w}} \leq k \right\}.
\end{equation}
Since the cone does not depend on $w$, we denote this by $ \mathcal{C}^{k}_{\delta}$. The next proposition shows its invariance by the transfer operator.
\begin{proposition} \label{invariancia} For every $w\in X$, there exists $0< \gamma_w < 1$ such that  
	$$ \mathcal{L}_{w}(\mathcal{C}^{k}_{\delta}) \subset  \mathcal{C}^{\gamma_{w}k}_{\delta} \subset \mathcal{C}^{k}_{\delta}$$
	 for some positive constant $k$ large enough.
\end{proposition}
\begin{proof}
	Given $\varphi\in \mathcal{C}^{k}_{\delta}$ we will show that 
	$$\frac{| \mathcal{L}_w(\varphi)|_{\alpha, \delta}}{\inf  \mathcal{L}_w(\varphi)}  \leq \gamma_w k  \quad \mbox{for some} \quad 0<\gamma_w <1.$$
	For each $x\in M$ and $1\leq j\leq \deg (f_w)$, denote by $x_j$ the preimage of $x$ under $f_w$. Observe that for any continuous function $\varphi$ we have 
	
	\begin{equation} \label{infimo}	\mathcal{L}_{w}(\varphi)(x) = {\displaystyle \sum_{j=1}^{{\deg(f_w)}}} e^{\phi_w(x_j)}\varphi(x_j)  
	\geq\deg(f_w)e^{\inf \phi_w} \inf\varphi . 
	\end{equation}
	From definition of $\mathcal{L}_w$ and the constant $|\mathcal{L}_w(\varphi)|_{\alpha, \delta}$ we obtain the following
	$$	\dfrac{|\mathcal{L}_w(\varphi)|_{\alpha, \delta}}{\inf \mathcal{L}_w(\varphi)}=\!\! \sup_{d(x,y)<\delta}\dfrac{|\mathcal{L}_w(\varphi(x)) - \mathcal{L}_w(\varphi(y))|}{\inf \mathcal{L}_w(\varphi) \, d(x,y)^{\alpha}}\leq \sum_{j=1}^{{\deg(f_w)}} \dfrac{\displaystyle \left| e^{\phi_w(x_j)}\varphi(x_j) -e^{\phi_w(y_j)} \varphi(y_j) \right| }{\inf \mathcal{L}_w(\varphi) \, d(x,y)^{\alpha}}$$
	By remark~(\ref{infimo}) the last inequality is less or equal than
	$$ \sum_{j=1}^{{\deg(f_w)}}\frac{e^{\sup\phi_w} \displaystyle  |\varphi(x_j) - \varphi(y_j)|}{\deg(f_w)e^{\inf \phi_w} \inf\varphi  \, d(x,y)^{\alpha}}+  \sum_{j=1}^{{\deg(f_w)}}  \dfrac{\sup \varphi\displaystyle  \left|e^{\phi_w(x_j)} - e^{\phi_w(y_j)}\right| }{\deg(f_w)e^{\inf \phi_w} \inf\varphi  \, d(x,y)^{\alpha}}.$$
	Recall that we are assuming that every point has $p_w$ preimages in the expanding region. Moreover, $\varphi$ is $((1+ (L_w-1)^\alpha)|\varphi|_{\alpha,\delta}, \alpha)$-H\"older continuous in balls of radius $L_w\delta$ we conclude that the previous sum is bounded from above by 
	\begin{eqnarray*}
		\frac{e^{\sup\phi_w} [p_w\sigma^{-\alpha}+ q_w L_w^{\alpha} (1+ (L_w-1)^\alpha ) ] \  |\varphi|_{\alpha,\delta} d(x,y)^{\alpha}  }{\deg(f_w)e^{\inf \phi_w} \inf\varphi  \, d(x,y)^{\alpha}}
		+  \frac{\sup \varphi|e^{\phi_{w}}|_{\alpha} L_w^\alpha d(x,y)^{\alpha}}{e^{\inf \phi_w} \inf\varphi  \, d(x,y)^{\alpha} }.
	\end{eqnarray*}
	Using equation~(\ref{supcone}), the definition of cone and condition~(\ref{cond1}) follows that the sum above is less or equal than
	$$\left[e^{\varepsilon_\phi}\left[\frac{ p_w\sigma^{-\alpha}+ q_w L_w^{\alpha} (1+ (L_w-1)^\alpha )}{\deg(f_w)}\right] + \varepsilon_\phi L_w^\alpha  \left[ 1+ m (\mbox{\mbox{diam}}M)^{\alpha}\right]\right]k. $$
	By hypotheses, condition~(\ref{condcone}), there exists some positive constant $0<\gamma_w< 1$ such that the previous sum is bounded from above by $\gamma_{w} k$. This finishes the proof.
	\end{proof}

From the last proposition we have the invariance of the cone $\mathcal{C}^{k}_{\delta}.$ Since this cone has finite diameter, according Proposition~\ref{diamfinito}, we can apply Theorem~\ref{contcone} to conclude the next result.	

\begin{proposition} \label{contracaocone}For every $w\in X$ the operator $\mathcal{L}_{w}$ is a contraction in the cone $\mathcal{C}^{k}_{\delta}$, i.e., denoting by $\Delta_w = \mbox{diam}_{\Theta_k}(\mathcal{C}^{\gamma_w k}_{\delta}) >0$ follows that
	$$\Theta_{k} \left(\mathcal{L}_{w}(\varphi), \mathcal{L}_{w}(\psi) \right) \leq (1-e^{-\Delta_w}) \cdot \Theta_k \left( \varphi, \psi \right) \quad \mbox{for all} \ \varphi, \psi \in \mathcal{C}^{k}_{\delta}.$$	
\end{proposition} Since we assume in condition (\ref{condcone}) the existence of $\gamma\in (0,1)$ such that $\gamma_w\leq \gamma $ for all $w\in X$ we conclude that  
$$\Theta_{k} \left(\mathcal{L}_{w}(\varphi), \mathcal{L}_{w}(\psi) \right) \leq (1-e^{-\Delta}) \cdot \Theta_k \left( \varphi, \psi \right) \quad \mbox{for all} \ \varphi, \psi \in \mathcal{C}^{k}_{\delta}\ \mbox{and} \ w\in X$$
where $\Delta= \sup_w (\Delta_w)\leq \mbox{diam}_{\Theta_k}(\mathcal{C}^{\gamma k}_{\delta})$

Let $\{\nu_w\}_w$ be the family of reference measures and $\lambda_w=\nu_{\theta(w)}(\mathcal{L}_{w}(1))$. The contraction in the cone allows us to prove the existence of the family $\{h_w\}_w$ invariant by the transfer operator.

\begin{proposition}\label{h} For almost $w\in X$ there exists a H\"older continuous function $h_{w}:M\to\mathbb{R}$ bounded away from zero and infinity satisfying 
	$\mathcal{L}_{w}h_{w}=\lambda_{w} h_{\theta(w)}.$
\end{proposition}

\begin{proof}
	Consider $\hat{\mathcal{L}}_w:=\lambda^{-1}_w\mathcal{L}_w$ the normalized operator and define the sequence $(\varphi_n)_n$ by $\varphi_n:=\hat{\mathcal{L}}^{n}_{\theta^{-n}(w)}(\mathbf{1})$ where 
	$$		\hat{\mathcal{L}}^{n}_{\theta^{-n}(w)}:= \hat{\mathcal{L}}_{\theta^{-1}(w)}\circ \hat{\mathcal{L}}_{\theta^{-2}(w)}\circ \cdots \circ \hat{\mathcal{L}}_{\theta^{-(n-1)}(w)}\circ \hat{\mathcal{L}}_{\theta^{-n}(w)}
	$$
	for each $n\geq 0$.  By definition of conformal measure we have 
	$$
	\int \varphi_n\ d\nu_{w} = \int \hat{\mathcal{L}}^{n}_{\theta^{-n}(w)}(\mathbf{1})\ d\nu_{w}
	=\int\mathbf{1}\ d(\hat{\mathcal{L}}^{\ast})^{n}_{\theta^{-n}(w)} \nu_w 
	= \int \mathbf{1}\ d \nu_{\theta^{-n}(w)}=1.
	$$
	Hence each term $\varphi_n$	satisfies $\sup \varphi_n\geq 1$ and $\inf \varphi_n\leq 1$. Since $\mathbf{1}\in \mathcal{C}^{k}_{\delta}$ and $\mathcal{C}^{k}_{\delta}$ is invariant, follows that $\varphi_n\in \mathcal{C}^{k}_{\delta} $ and so, applying inequality~(\ref{supcone}), we obtain that the sequence $(\varphi_n)_n$ is uniformly bounded away from zero and infinity by
	$$\frac{1}{R}\leq\inf\varphi_n \leq 1\leq \sup \varphi_n \leq R. $$
	where $R=(1 + mk \ diam(M)^\alpha)$.
	Moreover, as $\varphi_n$ is $C$-H\"older continuous in balls of radius $\delta$, by Lemma~\ref{bolas} we obtain that $\varphi_n$ is a $Cm$-H\"older continuous function. 
	
	Next we prove that $(\varphi_n)_n$ is a Cauchy sequence in the $C^0$-norm. From Proposition~\ref{contracaocone}, for every $m, l\geq n$ the projective metric satisfies 
	$$\Theta_{k}(\varphi_m, \varphi_l)=\Theta_{k}(\hat{\mathcal{L}}^{m}_{\theta^{-m}(w)}(\mathbf{1}), \hat{\mathcal{L}}^{l}_{\theta^{-l}(w)}(\mathbf{1})) \leq \Delta \tau^n \quad \mbox{where}\quad \tau:=  1 - e^{-\Delta}.  $$	 
	Recalling the expression of the projective metric $\Theta_{k}(\varphi_m, \varphi_l)= \log\left(\frac{B_k(\varphi_m, \varphi_l)}{A_k(\varphi_m, \varphi_l)}\right)$ 
	we apply  Lemma~\ref{metrica cone} to obtain
	\begin{eqnarray*} \label{cota sup}
		e^{-\Delta \tau^{n}}\leq A_{k}(\varphi_m, \varphi_l)\leq\inf \frac{\varphi_m}{\varphi_l}
		\leq1\leq \sup \frac{\varphi_m}{\varphi_l} \leq B_{k}(\varphi_m, \varphi_l) \leq e^{\Delta \tau^n}.
	\end{eqnarray*}
	Thus for all $m, l\geq n$, we have
	$$\left	\|\varphi_m - \varphi_l\right \|_{\infty} \leq \left\|\varphi_l \right\|_{\infty} \left\|\frac{\varphi_m}{\varphi_l} - 1 \right\|_{\infty} \leq R (e^{\Delta \tau^n}-1) \leq \tilde{R} \tau^{n}
	$$	
	which proves that $(\varphi_n)_n$ is a Cauchy sequence. Hence $(\varphi_n)_n$ converges uniformly to a function $h_w:M\to \mathbb{R}$ in the cone $\mathcal{C}^{k}_\delta$ satisfying $\int h_w d\nu_w =1$ for a.e. $w\in X$. In particular, this function is H\"older continuous and uniformly bounded away from zero and infinity. To complete the proof of the proposition, we are going to show that $\mathcal{L}_{w}h_{w}=\lambda_{w} h_{\theta(w)}.$ Consider the sequence $$\tilde{\varphi}_{n, w}:=\frac{1}{n}\sum_{j=0}^{n-1}\varphi_j=\frac{1}{n}\sum_{j=0}^{n-1}\hat{\mathcal{L}}^{j}_{\theta^{-j}(w)}(\mathbf{1}).$$
	By what we have proved above $(\tilde{\varphi}_{n, w})$ converges uniformly to $h_w$ for almost every $w\in X.$ From the continuity of $\mathcal{L}_w$ we obtain
	\begin{eqnarray*}
		\hat{\mathcal{L}}_w(h_w)=\lim_{n\to +\infty}\hat{\mathcal{L}}_w(\tilde{\varphi}_{n, w})&=&\lim_{n\to +\infty}\frac{1}{n}\sum_{j=0}^{n-1}\hat{\mathcal{L}}_w(\hat{\mathcal{L}}^{j}_{\theta^{-j}(w)}(\mathbf{1}))\\
		&=&\lim_{n\to +\infty}\frac{1}{n}\sum_{j=0}^{n-1}\hat{\mathcal{L}}^{j}_{\theta^{-j}(\theta(w))}(\mathbf{1})+\frac{1}{n}(\hat{\mathcal{L}}^n_{\theta^{-n}(\theta(w))}(\mathbf{1})-\mathbf{1})
	\end{eqnarray*}
	Since $\hat{\mathcal{L}}^n_{\theta^{-n}(\theta(w))}(\mathbf{1})$ is uniformly bounded we conclude that $\hat{\mathcal{L}}_w(h_w)=h_{\theta(w)}.$
\end{proof}

From the proof of the last proposition we conclude that the family $\{h_w\}_{w\in X}$ is uniquely determined. Moreover, every $h_w$ satisfies
$$\frac{1}{R}\leq\inf h_w \leq 1\leq \sup h_w \leq R. $$
where  $R=(1 + mk \ diam(M)^\alpha)$.
	
	\subsection{Mensurability} \label{mensu} 
In the last subsection we have proved the existence of a family $\{h_w\}_w$ invariant under the action of the transfer operator. Here we prove that this family is measurable as well as the family $\{\nu_w\}_w$. Moreover, defining the probability measure $\mu_w:=h_w\nu_w$ we also prove that $\mu_w$ has an exponential decay of correlations and that the family $\{\mu_w\}_w$ is $F$-invariant.

The next proposition states an exponential approximation of functions in the cone to the invariant family $\{h_w\}_w$. This is the main ingredient in the proof of the exponential decay of correlations.

\begin{proposition} \label{cota norma}
	For almost $w\in X$ there exist constants $K>0$ and $0<\tau <1$ such that for every $\varphi\in \mathcal{C}^{k}_{\delta}$ satisfying $\int \varphi \ d\nu_w = 1$ we have that
	$$\left\| \mathcal{\hat{L}}_w^{n}(\varphi) - \mathcal{\hat{L}}_w^{n}(h_w) \right\|_{\infty} \leq K\tau^n\,,   \quad \mbox{for all}\quad n\geq 1,$$ 
	where $\mathcal{\hat{L}}_w=\lambda_{w}^{-1}\mathcal{L}_w$ is the normalized operator.
\end{proposition}

\begin{proof}
	Given $\varphi \in \mathcal{C}^{k}_{\delta}$ satisfying $\int \varphi \ d\nu_w = 1$ we have for every $n \geq 1$
	$$ \int \mathcal{\hat{L}}_w^{n}(\varphi) \ d\nu_{\theta^{n}(w)} = \int \varphi \ d(\mathcal{L}^{\ast}_{w})^{n}(\nu_{\theta^{n}(w)}) = \int \varphi\ d\nu_w =1 .$$
	Since $h_w \in \mathcal{C}^{k}_{\delta}$ also satisfies $\int h_w \ d\nu_w = 1$, we derive for every $n\geq 1$ that 
	$$\inf\frac{\mathcal{\hat{L}}_w^{n}(\varphi)}{\mathcal{\hat{L}}_w^{n}(h_w)} \leq 1\leq \sup \frac{\mathcal{\hat{L}}_w^{n}(\varphi)}{\mathcal{\hat{L}}_w^{n}(h_w)}.$$		
	Recalling that $\mathcal{\hat{L}}^n_{w}(h_w)=h_{\theta^{n}(w)}$ for all $n\geq 1$, we apply the same argument of projective metric used in the proof of Proposition~\ref{h} to obtain
	$$\left\| \mathcal{\hat{L}}_w^{n}(\varphi) - \mathcal{\hat{L}}_w^{n}(h_w) \right\|_{\infty} \leq \left\|h_{\theta^{n}(w)} \right\|_{\infty} \left\| \frac{\mathcal{\hat{L}}_w^{n}(\varphi)}{\mathcal{\hat{L}}_w^{n}(h_w)} -1 \right\|_{\infty} \leq R \left(e^{\Delta\tau^n} -1 \right) \leq K\tau^n. $$
\end{proof}

Let $\mu_w$ be the probability measure defined by $\mu_w:=h_w\nu_w.$ From the last proposition we derive the proof of Theorem~\ref{decaimento}.

\begin{thB}
	\label{decaimento 2}
	For a.e. $w\in X,$ the probability measure $\mu_{w}$ has exponential decay of correlations for H\"older continuous observables: There exists $0<\tau< 1$ such that for any $\varphi\in L^1(\mu_{\theta^{n}(w)})$ and $\psi \in C^{\alpha}(M) $ there exists a positive constant $K(\varphi, \psi) $ satisfying for all $n\geq 1$ that
	$$\left|\int \left(\varphi \circ f^{n}_{w}\right) \psi \ d\mu_{w} - \int \varphi \ d\mu_{\theta^{n}(w)} \int \psi\ d\mu_{w}  \right| \leq K(\varphi, \psi) \tau^n.$$

\end{thB}	

\begin{proof}
	Given $\varphi \in L^1(\mu_{\theta^{n}(w)})$ and $\psi \in C^{\alpha}(M)$ we supppose without loss of generality that $\int \psi \  d\mu_{w} = 1$. Let $\hat{\mathcal{L}}_w=\lambda^{-1}_w\mathcal{L}_w$ be the normalized operator. As a first case we consider $\psi \cdot h_w$ in the cone $\mathcal{C}^{k}_{\delta}$ for $k$ large enough. Recalling that $\mu_w=h_w\nu_w$ and that $\mathcal{L}^{*}_{w}\nu_{\theta(w)}=\lambda_w\nu_w$ we have 
	\begin{eqnarray*}
		&&\left|\int\! \left(\varphi \circ f^{n}_{w}\right) \psi \ d\mu_{w} - \int\!\varphi\ d\mu_{\theta^{n}(w)} \int\! \psi\ d\mu_{w}  \right|\\
		&\!\!=\!\!& \left|\int\! \left(\varphi \circ f^{n}_{w}\right) \psi\cdot  h_{w} \ d\nu_{w} - \int\! \varphi\cdot h_{\theta^{n}(w)} \ d\nu_{\theta^{n}(w)}  \right|\\
		&\!\!=\!\!&\left| \int\! \varphi \cdot \mathcal{\hat{L}}_{w}^n\left( \psi \cdot h_w \right) \ d\nu_{\theta^{n}(w)} -  \int \varphi\cdot h_{\theta^{n}(w)} \ d\nu_{\theta^{n}(w)} \right|\\
		&\!\!\leq\!\!& \|\varphi\|_1\left\| \mathcal{\hat{L}}_{w}^n\left( \psi \cdot h_w \right) -\mathcal{\hat{L}}_{w}^n\left(h_w \right)  \right\|_\infty.
	\end{eqnarray*}
	
	Since $\psi\cdot h_w\in \mathcal{C}^{k}_{\delta}$ and $\int \psi\cdot h_w \  d\nu_{w}=\int \psi\  d\mu_{w} = 1$ we can apply Proposition~\ref{cota norma} to conclude the existence of constants $K>0$ and $0<\tau<1$ such that 
	$$\left\| \mathcal{\hat{L}}_{w}^n\left( \psi \cdot h_w \right) -\mathcal{\hat{L}}_{w}^n\left(h_w \right)  \right\|_\infty\leq  K \tau^ n\,\,\,\mbox{for every}\,\,\, n\geq 1.$$
	For the general case we write $\psi\cdot h_w= g$ where
	$$g=g^{+}-g^{-}\quad;\quad g^{\pm}=\frac{1}{2}\left(|g|\pm\psi\right)+C\quad \mbox{and}\quad C=k^{-1}|\psi \cdot h_w|_{\alpha, \delta}.$$
	Therefore $g^{\pm}\in \mathcal{C}^{k}_{\delta}.$ From the previous estimates on $g^{\pm}$ and by linearity the proposition holds.
\end{proof}

Now we state the mensurability of the families $\{\nu_w\}_w$ and $\{h_w\}_w$. We start by observing that for a.e. $w\in X$ and every continuous function $g\in C^0(M)$ we have 
\begin{equation}\label{deccontinua}
\frac{1}{h_{\theta^{n}(w)}}\,\mathcal{\hat{L}}^{n}_w(g\cdot h_w)\longrightarrow \int g\cdot h_w\, d\nu_w\quad \mbox{when}\quad n\to\infty
\end{equation}
where $\hat{\mathcal{L}}$ is the normalized operator $\hat{\mathcal{L}}_w = \lambda^{-1}_w \mathcal{L}_w.$ Indeed, we can suppose that the function $g_w$ is H\"older continuous because any continuous function is aproximated by such functions. Moreover, following the proof of Theorem~\ref{decaimento 2}, we can just consider the case $ g \cdot h_w \in \mathcal{C}^{k}_{\delta}$ for $k$ large enough.  We have that
\begin{eqnarray*}
	\left\|\frac{1}{h_{\theta^{n}(w)}}\,\mathcal{\hat{L}}^{n}_w(g h_w) - \int g h_w\, d\nu_w \right\|_\infty 
	&\!\!\!\leq\!\!\! &\left\|\dfrac{1}{h_{\theta^{n}(w)} }\right\| \left\| \mathcal{\hat{L}}^{n}_w(g h_w) - \int g h_w\, d\nu_w\cdot h_{\theta^{n}(w)} \right\| \\
	&\!\!\!\leq\!\!\! &\displaystyle R \left\|g h_w\right\|_\infty \left\| \mathcal{\hat{L}}^{n}_w\left(\frac{gh_w}{\int g h_w\, d\nu_w}\right) - \mathcal{\hat{L}}^{n}_w(h_{w}) \right\|_\infty\\
\end{eqnarray*}
By Proposition~\ref{cota norma} 
the convergence in (\ref{deccontinua}) follows. We use this on the next result.			

\begin{lemma}\label{unicidade familia} Let $\lambda_{w}=\nu_{\theta(w)}(\mathcal{L}_{w}(1)).$ The family $\{\nu_w\}$ is uniquely determined by $$
	\mathcal{L}_{w}^{\ast}\nu_{\theta(w)}=\lambda_{w}\nu_{w.}$$
	 Moreover, the map $w\mapsto \nu_w(g_w)$ is measurable for any $g\in \mathbb{L}^{1}_{\mathbb{P}}(X, C^0(M))$.
\end{lemma}
\begin{proof}
	Fix $w\in X$ and let $(x_n)$ be a sequence of points in $M$. Define the probability $$\nu_{w, n}=\frac{(\mathcal{L}^{n}_w)^{\ast}\delta_{x_{n}}}{\mathcal{L}^{n}_{w}1(x_n)}.$$
	Since $\nu_w$ satisfies the condition $\mathcal{L}_{w}^{\ast}\nu_{\theta(w)}=\lambda_{w}\nu_{w} $ we can apply the convergence (\ref{deccontinua}) to conclude that for a.e. $w\in X$ and any continuous function $g_w$ we have 
	\begin{eqnarray*}\lim_{n\to\infty}{\nu_{w, n}(g_w)}=\lim_{n\to\infty}\frac{(\mathcal{L}^{n}_w)^{\ast}\delta_{x_{n}}(g_w)}{\mathcal{L}^{n}_{w}1(x_n)}\!\!\!\!&=&\!\!\!\!\lim_{n\to\infty}\frac{\mathcal{L}^{n}_{w}g_{w}(x_n)}{\mathcal{L}^{n}_{w}1(x_n)}\\ \\
		&=&\displaystyle\lim_{n\to\infty}\frac{\mathcal{L}^{n}_{w}\left(\displaystyle\frac{g_{w}}{h_w}\cdot h_w\right)(x_n)}{\mathcal{L}^{n}_{w}\left(\displaystyle\frac{1}{h_w}\cdot h_w\right)(x_n)}=\nu_{w}(g_w).
	\end{eqnarray*}  	
	Therefore follows the convergence of $\nu_{w,n}\stackrel{w*}\longrightarrow\nu_{w}$. Since the sequence $(x_n)$ was arbitrary the uniqueness of the family $\{\nu_w\}$ is proved.
	Moreover, the equality
	$$\lim_{n\to\infty}\frac{\|\mathcal{L}^{n}_{w}g_w\|_{\infty}}{\|\mathcal{L}^{n}_{w}1\|_{\infty}}=\nu_w(g_w).$$
	implies the mensurability of $w\mapsto \nu_w(g_w)$ since the transfer operator is measurable.
\end{proof}

The lemma above enables us to define the probability measure $\nu$ on the Borelean sets of $X\times M$ by the following 
$$\nu(g)=\int_{X}\, \int_{M}\,  g_w\, d\nu_w\,d\mathbb{P}(w).$$

Let $c>0$ given by condition~(\ref{cond2}) and consider $H\subset X\times M$ be the non-uniformly expanding set defined in~(\ref{H}). As in Section~\ref{medreferencia} for $w\in X$ consider 
$$H_w:=\left\{x\in M\ \ ;\ \ \limsup_{n\to+\infty} \frac{1}{n}\sum^{n-1}_{j=0}\log L_{\theta^{j}(w)}(f^{j}_{w}(x))^{-1}\leqslant -2c<0\right\}.$$
By Proposition~\ref{expanding} for almost every $w\in X$ we have $\nu_w(H_w)=1$. Thus we conclude that $\nu(H)=1$, i.e., $\nu$ is a non-uniformly expanding measure. 

Notice that since $\lambda_{w}=\nu_{\theta(w)}(\mathcal{L}_{w}(1))$ we have that the map $w\mapsto \lambda_w\in \mathbb{R}$ is measurable. Recalling that for almost every $w\in X$ the function $h_w$ is given by $$h_w=\lim_{n\to\infty}{\mathcal{\hat{L}}^{n}_{\theta^{-n}(w)}1}$$ we deduce the measurability of the map $(w, y)\mapsto h_w(y)$ from the measurability of $\lambda_w$ and the transfer operator. From this results it is well defined the probability measure $\mu_{F, \phi}\in\mathcal{M}_{\mathbb{P}}(X\times M)$ by the formula 
$$\mu_{F, \phi}(g)=\int_{X}\, \int_{M} \, g_w\cdot h_w\, d\nu_w\, d\mathbb{P}(w).$$
In order to prove the $F$-invariance of $\mu_{F, \phi}$ we first observe that 
\begin{eqnarray*}
	\mu_w (g_{\theta(w)}\circ f_w) =\!\! \int g_{\theta(w)}\circ f_w\cdot h_w\ d\nu_w
	\!\!\!\!\!&=&\!\!\!\!\!\int \hat{\mathcal{L}}_w (g_{\theta(w)}\circ f_w \cdot h_w) \ d\nu_{\theta(w)}\\
	&=&\!\!\! \int\dfrac{g_{\theta(w)}\cdot \hat{\mathcal{L}}_w (h_w)}{h_{\theta(w)}} \ d\mu_{\theta(w)}
	= \mu_{\theta(w)}(g_{\theta(w)}).
\end{eqnarray*}	
Therefore for every integrable function we get
\begin{align*}
\int\!\! g \circ F d\mu_{F, \phi} = \int\!\!\!  \int\!\! g_{\theta(w)}\circ f_w(x) \ d\mu_w d\mathbb{P}(w)
 = \int\!\!\! \int\!\!g_{\theta(w)} d\mu_{\theta(w)} d\mathbb{P}(w) = \int\!\! g\  d\mu_{F, \phi}.
\end{align*}
This finishes the proof of Theorem~\ref{formalismo}.

\section{Equilibrium States}\label{ee}
In this section we prove that the measure $\mu_{F, \phi}$ constructed above is an equilibrium state for $(F|_{\theta}, \phi)$. Moreover, we show that any ergodic non-uniformly expanding equilibrium state has disintegration absolutely continuous with respect to the system of reference measures. From this we derive the uniqueness.	

As a first step, in the next proposition we obtain an upper bound for the topological pressure of the random dynamical system.

\begin{proposition} \label{pressao} For any potential $\phi$ satisfying condition~(\ref{cond1}) we have that $$P_{F|\theta}(\phi)\leq\displaystyle \int_X \log \lambda_w\ d\mathbb{P}(w).$$
\end{proposition}

\begin{proof}
	Fix $w\in X$ such that $H_w$ is not empty (see definition in Section~\ref{medreferencia}) and let $\varepsilon>0$ be small. Since every point $x\in H_w$ has infinitely many hyperbolic times, for $N>1$ large enough we have 
	$$H_w\subset\bigcup_{n\geq N}\bigcup_{x\in H_n}B_{w}(x,n, \varepsilon),$$
	where $H_n=H_n(w)$ denotes the set of points that have $n$ as a hyperbolic time. 
	From Lemma~\ref{lemacontra} each $f^n(B_w(x,n, \varepsilon))$ is the ball $B_{\theta^n(w)}(f_w^n(x),\varepsilon)$ in $M$, thus by apply Besicovitch Covering Lemma it is straigthforward to check that there exists a countable family $F_{n}\subset H_{n}$ such that every point $x\in H_{n}$ is covered by at most $d=d(\dim(M))$ dynamical balls $B_{w}(x,n, \varepsilon)$ with $x\in F_{n}.$ Therefore 
	$$\mathcal{F}_{N}=\left\{B_{w}(x,n, \varepsilon) : x\in F_{n}\;\textbf{\rm{and}}\;n\geq N\right\}$$ 
	is a countable open covering of $H_w$ by dynamic balls with diameter less than $\varepsilon>0$. Actually $\mathcal{F}_{N}$ is an open cover of $M$ because $H_w$ is dense.
	
	Let $\beta > \int\log\lambda_w\ d\mathbb{P}(w).$ By the definition of topological pressure given in Section~\ref{preliminares} and by applying Lemma~\ref{conforme} to each element in $\mathcal{F}_{N}$ we obtain the following
	\begin{eqnarray*}
		m_{\beta}(w, \phi, F|_{\theta}, \varepsilon, N)&\leq& \sum_{B_{w}(x,n, \varepsilon)\in \mathcal{F}_{N}}   e^{-\beta n+S_n\phi(B_{w}(x,n, \varepsilon))}\\
		&\leq&\sum_{n\geq N}\gamma^{-1}_{\varepsilon}(\theta^n(w))K_\varepsilon(w)e^{-(\beta n-\log\lambda_w^n)}\sum_{x\in F_n}\nu_w(B_{w}(x,n, \varepsilon))\\
		&\leq& d K_\varepsilon(w)\sum_{n\geq N}\gamma^{-1}_{\varepsilon}(\theta^n(w))e^{-(\beta-\frac{1}{n}\sum_{i=0}^{n-1}\log\lambda_{\theta^{i}(w)})n}.
	\end{eqnarray*} 	
	
	As from Remark~(\ref{K}) the variable $\gamma^{-1}_{\varepsilon}(\theta^n(w))$ is uniformly bounded and by ergodicity we have $\lim_{n\to \infty}\frac{1}{n}\sum_{i=0}^{n-1}\log\lambda_{\theta^{i}(w)}=\int\log\lambda_w\ d\mathbb{P}(w)$ for almost $w\in X$, we obtain when $N$ goes to infnity that
	$$m_{\beta}(w,\phi, F|_{\theta}, \varepsilon)=\lim_{N\rightarrow +\infty}{m_{\beta}(w, \phi, F|_{\theta}, \varepsilon, N)}=0$$
	for any $\beta>\int\log\lambda_w\ d\mathbb{P}(w)$ and $\varepsilon>0$ small. Thus we necessarily have $$P_{F|_{\theta}}(w, \phi)\leq\int_X\log\lambda_w\ d\mathbb{P}(w).$$
	Since this is true for $\mathbb{P}$-almost every $w\in X$, we prove the proposition.
\end{proof}

In the next subsection we are going to prove that $P_{F|_{\theta}}(\phi)=\int\log\lambda_w\ d\mathbb{P}(w).$

\subsection{Existence} Consider $\{\nu_w\}_{w\in X}$ the family of reference measures and $\{h_w\}_{w\in X}$ as in Theorem~\ref{formalismo}.
For each $w\in X$ let $\mu_w$ be the probability measure defined by $\mu_w=h_w\nu_w$.
Recalling that the jacobian of $\nu_w$ is  $J_{\nu_w}f_w=\lambda_{w}e^{-\phi_{w}}$ it is easy to verify that the jacobian of $\mu_w$ relative to $f_w$ is given by 
\begin{equation} \label{jacobiano}
J_{\mu_{w}}f_w=\frac{\lambda_{w}e^{-\phi_{w}}h_{\theta(w)}\circ f_{w}}{h_w}.
\end{equation}	
Consider $\mu_{F, \phi}$ the probability measure which desintergration is $\{\mu_w\}_w$ that means
$$\mu_{F, \phi}(g)=\int_{X}\, \int_{M} \, g_w\cdot h_w\, d\nu_w\, d\mathbb{P}(w)$$
for every continuous function $g:X\times M\to \mathbb{R}$. As we seen in the Subsection~\ref{mensu}, $\mu_{F,\phi}$ is $F$-invariant. In the next proposition we show its ergodicity.

\begin{proposition} The probability measure $\mu_{F, \phi}$ is ergodic.
\end{proposition}
\begin{proof}
	Given a $F$-invariant set $A\subset X\times M$, for each $w\in X$ denote by $A_{w}$ the set $A_w=\{z\in M\,;\, (w, z)\in A\}.$ The $F$-invariance of $A$ implies that $f_{w}^{-1}(A_{\theta(w)})=A_w.$ Consider $X_{0}=\{w\in X\,;\,\mu_{w}(A_w)>0\}$. It is straightforward to check that $X_0$ is a $\theta$-invariant subset of $X$. Since $\theta$ is ergodic with respect to $\mathbb{P}$, we will obtain the ergodicity of $\mu_{F, \phi}$ by showing that for almost every $w\in X_0$ we have $\mu_{w}(A_w)=1$, if $\mathbb{P}(X_0)>0.$
	
	Let $\varphi_w$ be the characteristc function of $A_w$, i.e., $\varphi_w=\mathbf{1}_{A_w}$. Notice that in $\mathbb{P}$-a.e. holds $\varphi_{\theta^{n}(w)}\circ f_{w}^{n}=\varphi_w$.  Given $\psi_w\in L^{1}(\mu_w)$ such that $\int \psi_w\, d\mu_w=0$, from the decay correlation property of $\mu_w$, Theorem~\ref{decaimento}, follows that $$\mu_w\bigl(\bigl(\varphi_{\theta^{n}(w)}\circ f_{w}^{n}\bigr)\cdot \psi_w\bigr)\rightarrow 0\,\, \mbox{when}\,\, n\to+\infty.$$	And thus $\int_{A_w} \psi_w\, d\mu_w=0$ for any $\psi_w\in L^{1}(\mu_w)$ satisfying $\int \psi_w\, d\mu_w=0$. This proves that $\mu_w(A_w)=1$ for $\mathbb{P}$-almost every $w\in X_{0}$ which finishes the proof.
\end{proof}

In Subsection~\ref{mensu} we observe that the measure $\nu$ defined by 
$$\nu(g)=\int_{X}\, \int_{M}\,  g_w\, d\nu_w\,d\mathbb{P}(w)$$
is non-uniformly expanding. Since $\mu_{F, \phi}$ is absolutely continuous with respect to it, we have that $\mu_{F, \phi}$ is also a non-uniformly expanding measure. In particular, by Lemma~\ref{gp}, it admits a generating partition. Thus we can use the random Rokhlin's formula to express the entropy of $\mu_{F, \phi}$ in terms of its jacobian.

	\begin{theorem}[Random Rokhlin's formula]
		\label{formulaRokhlin}
		
		Let $\mu\in \mathcal{M}_{\mathbb{P}}(F)$ be an ergodic measure which admits a $\mu$-generating partition. Then 
		\[
		h_{\mu}(F|\theta)= \int \log J_\mu (F) \,d\mu =  \int_{X}\biggl( \int_{M}\log J_{\mu_w}f_{w}(y) d\mu_{w}(y)\biggr) \, d\mathbb{P}(w),
		\]
		where $J_{\mu_{w}}f_{w}$ denotes the jacobian of $f_{w}$ relative to $\mu_{w}$.
	\end{theorem}

	Now we are read to prove that $\mu_{F, \phi}$ is an equilibrium state for $(F|\theta, \phi).$
			\begin{eqnarray*}
		h_{\mu_{F, \phi}}(F|\theta)& =&  \int_{X}\int_{M}\log\, J_{\mu_{w}}f_{w}(y)\, d\mu_{w}(y)\, d\mathbb{P}(w) \\ 
		&=& \int_{X} \int_{M}\log \biggl(\frac{\lambda_{w}e^{-\phi_{w}}h_{\theta(w)}\circ f_{w}}{h_w}\biggr)(y)\, d\mu_{w}(y) \, d\mathbb{P}(w)\\ 
		& =& \int_{X} \int_{M} \log\lambda_{w} d\mu_{w} (y) d\mathbb{P}(w) - \int_{X} \int_{M} \phi_{w}\, d\mu_{w} (y) d\mathbb{P}(w)\\ 
		&+&\int_{X} \int_{M} \biggl(\log h_{\theta(w)}\circ f_{w}(y) - \log h_{w}(y)\biggr)\, d\mu_{w} (y) d\mathbb{P}(w)
	\end{eqnarray*}
From the $F$-invariance of $\mu_{F, \phi}$ we derive that
$$\int_{X} \int_{M} \biggl(\log h_{\theta(w)}\circ f_{w}(y) - \log h_{w}(y)\biggr)\, d\mu_{w} (y) d\mathbb{P}(w)=0$$
 Thus we can write 
 \begin{eqnarray*}
		h_{\mu_{F, \phi}}(F|\theta)&=& \int_{X} \int_{M} \log\lambda_{w} d\mu_{w} (y) d\mathbb{P}(w) - \int_{X} \int_{M} \phi_{w}\, d\mu_{w} (y) d\mathbb{P}(w)\\ \\&= &\int_{X}\log \lambda_w\, d\mathbb{P}(w) - \int \phi\, d\mu_{F, \phi}
		\end{eqnarray*}
Applying the variational principle~(\ref{eqprinvaria}) and  Proposition~\ref{pressao} follows that 
$$\int_{X}\log \lambda_w\, d\mathbb{P}(w)=h_{\mu_{F, \phi}}(F|\theta)+\int \phi\, d\mu_{F, \phi}\leq P_{ F|\theta}(\phi)\leq\int_{X}\log \lambda_w\, d\mathbb{P}(w)$$
which implies in particular that $P_{ F|\theta}(\phi)=\displaystyle\int_{X}\log \lambda_w\, d\mathbb{P}(w)$ and so $\mu_{F, \phi}$ is an equilibrium state.

	\subsection{Uniqueness} Until now we have proved the existence of an equilibrium state for $(F|_{\theta}, \phi)$. Here we prove the uniqueness in the set of non-uniformly expanding measures.

Let $\eta$ be an ergodic non-uniformly expanding equilibrium state for $(F|_{\theta}, \phi)$. We prove that the disintegration of $\eta$ is absolutely continuous to the reference measure. For this we use the following remark from the basic Calculus.

\begin{remark} [Jensen's Inequality]\label{Jensen} Given positive numbers $p_i>0$ and $q_i>0$, $i=1, \cdots, n$ such that $\sum_{i=1}^{n}p_i=1$ we have that $\sum_{i=1}^{n}p_i\log q_i\leq\log \left(\sum_{i=1}^{n}p_i q_i\right)$ and the equality holds if and only if the $q_i$ are equal.
\end{remark}

\begin{proposition} 
\label{propestadodeequilibrio}
Consider $\eta\in \mathcal{M}_{\mathbb{P}}(F)$ any ergodic non-uniformly expanding equilibrium state of $(F|_{\theta}, \phi)$ and let $(\eta_{w})_w$ its disintegration. Then for almost $w\in X$, $\eta_w$ is absolutely continuous with respect to $\nu_w$.
\end{proposition}

\begin{proof} First we prove that for almost every $w\in X$ the jacobian of $\eta_w$ is given by $$J_{\eta_{w}}f_w=\frac{\lambda_w e^{-\phi_w}\cdot h_{\theta(w)}\circ f_w}{h_w}.$$
	Indeed, since $\eta$ is an ergodic non-uniformly expanding equilibrium state we can apply the Rokhlin's formula to obtain that 
	$$h_{\eta}(F|\theta)= \int_{X}\biggl( \int_{M}\log J_{\eta_w}f_{w}(y)\, d\eta_{w}(y)\biggr)\, d\mathbb{P}(w)=\int_{X\times M}\log J_{\eta_w}f_{w}(y)\, d\eta(w, y)$$
	
	Recalling that $h_{w}$ is a bounded function and that the jacobian of $\mu_{w}$ is given by $J_{{\mu}_{w}}{f_w}\;=\lambda_w e^{-{\phi_w}}\cdot{h_{\theta(w)}}\circ{f_w}/{h_w}$ we have 
	\begin{eqnarray*}
		\int{\log\frac{J_{\eta_w}f_w}{J_{\mu_{w}}f_w}}\,d\eta(w, y)\!\!\!&=&\!\!\!\int{\log J_{\eta_w}f_w}\,d\eta-\int{\!\!\left(\log\lambda_w-\phi_w+\log\frac{h_{\theta(w)}\circ{f_w}}{h_w}\right)}\,d\eta\\
		&=&\!\!\!h_{\eta}(F|\theta)\!-P_{F|\theta}(\phi)+\!\int{\!\phi_w+ \log h_w-\log h_{\theta(w)}\!\!\circ\! f_w}\, d\eta \geq 0.
		\end{eqnarray*}
	
	From the definition of jacobian we can write
	\begin{equation} \label{defjacobiano}
		\int{\sum_{z=f_w^{-1}(y)}{J_{\eta_w}f_w^{-1}(z)}\log\frac{J_{\eta_w}f_w}{J_{\mu_{w}}f_w}(z)}d\eta(w, y)=\int{\log\frac{J_{\eta_w}f_w}{J_{\mu_{w}}f_w}(y)}\,d\eta(w, y)\geq 0.
	\end{equation}
		Take $p_i=J_{\eta_w}f_w^{-1}(z_i)$ and $q_i=J_{\eta_w}f_w(z_i)/J_{\mu_{w}}f_w(z_i)$ where $z_i$ are the pre-images of $y.$ Since $(f_w)_{\ast}\eta_w=\eta_{\theta(w)}$ we have $\sum_{i=1}^{\deg(f_w)}{p_i}=\sum_{z=f_w^{-1}(y)}J_{\eta_w}f_w^{-1}(z)=1$
	for $\eta_{\theta(w)}$ almost $y\in M.$ Therefore we can apply Remark~\ref{Jensen} to conclude that
	\begin{eqnarray*}
		\sum_{z=f_w^{-1}(y)}{J_{\eta_w}f_w^{-1}(z)}\log\frac{J_{\eta_w}f_w}{J_{\mu_{w}}f_w}(z)
		&\leq&\log\!\left(\sum_{z=f_w^{-1}(y)}\frac{J_{\eta_w}f_w^{-1}\cdot {J_{\eta_w}f_w}}{J_{\mu_{w}}f_w}\right)(z)\\
			&=&\log\left(\frac{\sum_{z=f_{w}^{-1}(y)}e^{\phi_w(z)}h_w(z)}{\lambda_w h_{\theta(w)}\circ f_w(z)}\right)\\
		&=&\log\left(\frac{\lambda_w h_{\theta(w)}(y)}{\lambda_w h_{\theta(w)}(y)}\right)=0.
	\end{eqnarray*}	
	for $\eta_{\theta(w)}$ almost $y\in M.$ Recalling the inequality (\ref{defjacobiano}) we obtain that 
		$$0\leq\int{\!\!\log\frac{J_{\eta_w}f_w}{J_{\mu_{w}}f_w}(y)}\,d\eta(w, y)\!=\!\int{\!\!\!\sum_{z=f_w^{-1}(y)}{J_{\eta_w}f_w^{-1}(z)}\log\frac{J_{\eta_w}f_w}{J_{\mu_{w}}f_w}(z)}d\eta(w, y)=0.$$
		Thus, from the second part of Remark~\ref{Jensen}, the values $q_i=J_{\eta_w}f_w(z_i)/J_{\mu_{w}}f_w(z_i)$ must be the same for all $z_i\in f_{w}^{-1}(y)$ in a full $\eta_{\theta(w)}$-measure set. In other words for every $y\in M$ on the pre-image of a full $\eta_{\theta(w)}$-measure set holds $$J_{\eta_w}f_w(y)=J_{\mu_{w}}f_w(y)=\frac{\lambda_w e^{-\phi_w}\cdot h_{\theta(w)}\circ f_w}{h_w}(y).$$ 
	
	To finish the proof of the proposition we observe that $\frac{1}{h_{w}}\cdot\eta_w$ is a reference measure associated to $\lambda_w$ for the dual transfer operator: 
	\begin{eqnarray*}
		\mathcal{L}^{\ast}_{w}\left(\frac{1}{h_{\theta(w)}}\cdot\eta_{\theta(w)}\right)(\psi)&=&\int\mathcal{L}_{w}(\psi)(x)\, d\left(\frac{1}{h_{\theta(w)}}\cdot\eta_{\theta(w)}\right)\\
		&=&\int{\sum_{y=f_w^{-1}(x)}e^{\phi_w(y)}(y)}\psi(y)\cdot\frac{1}{h_{\theta(w)}}(x)\,d\eta_{\theta(w)}\\
		&=&\!\!\!\int{\sum_{y=f_w^{-1}(x)}\lambda_w\frac{\psi(y)}{h_w(y)}\left(\frac{h_w(y)}{\lambda_w e^{-\phi_w(y)}\cdot h_{\theta(w)}\circ f_w(y)}\right)}\,d\eta_{\theta(w)}\\
		&=&\int{\sum_{y=f_w^{-1}(x)}\lambda_w\frac{\psi(y)}{h_w(y)} J_{\eta_w}f_w^{-1}(y)}\,d\eta_{\theta(w)}\\
		&=&\lambda_w\int{\psi}\,d\left(\frac{1}{h_{w}}\cdot\eta_w\right)
			\end{eqnarray*}
	
	From the uniqueness given by Theorem~\ref{formalismo} we conclude that $\frac{1}{h_{w}}\cdot\eta_w$ is equivalent to $\nu_w$ and thus $\eta_w$ is absolutely continuous to this one.
\end{proof}

Finally we prove the uniqueness of the equilibrium state associated to $(F|_{\theta}, \phi).$ Suppose that there exist $\mu$ and $\eta$ two ergodic equilibrium states. Let $(\mu_w)_{w\in X}$ and $(\eta_w)_{w\in X}$ be the disintegration of $\mu$ and $\eta$, respectively. 

By the proposition above we have that $\mu_{w}$ and $\eta_{w}$ are equivalent measures. From the Radon-Nykodym theorem we know that there exist a mensurable function $q_{w}: M\rightarrow \mathbb{R}$ such that $\mu_{w}=q_{w}\eta_{w}$ for every $w\in X$. Consider $q: X\times M \rightarrow \mathbb{R}$ defined by $q(w,x)=q_{w}(x)$. Given a measurable set $E\subset X\times M$  consider $E_{w}\subset M$  the instersection $E_{w}=E\cap M$. Then we have that
\begin{align*}
\mu(E)=\int_{X}\mu_{w}(E_{w})d\mathbb{P}(w)=\int_{X}\int_{E_{w}}q_{w}d\eta_{w}d\mathbb{P}(w)=\int_{E}q \ d\eta.
\end{align*}
Moreover from the $F$-invariance of $\mu$ and $\eta$ follows that 
$$	\mu(E)=F_{\ast}\mu(E)=(q\circ F)F_{\ast}\eta(E)=(q\circ F)\eta(E).
$$

Since the Radon-Nykodym derivative is essentially unique, we conclude that $q=q\circ F$ at $\eta$-almost every point. By ergodicity we have that $q$ is constant everywhere and thus $\mu=\eta.$ 

\subsection{Positive Lyapunov exponents} The main tool in the proof of Proposition~\ref{propestadodeequilibrio} is the existence of a generating partition for the equilibrium state. In the context of random dynamical systems generated by non-uniformly expanding maps, the existence of generating partitions for ergodic measures with Lyapunov exponents bounded away from zero was proved by Bilbao and Oliveira in \cite{Bilbao}. Therefore  in this setting we can also apply our Proposition~\ref{propestadodeequilibrio} to obtain uniqueness of equilibrium states. This is our goal now.

Consider $M^d$ a compact and connected Riemann manifold of dimension $d$. Let $F:X\times M\to X\times M$ be the skew-product $(\theta(w), f_w(x))$ generated by $C^1$ local diffeomorphisms $f_w: M\to M$ satisfying (I)-(III). For $1\leq k\leq d-1$ define $$
C_{k}(w,x)=\limsup_{n\rightarrow +\infty} \frac{1}{n}\log \|\Lambda^{k}Df_{w}^{n}(x)\|
\quad \mbox{and} \quad C_{k}(w,F)=\max_{x \in M}C_{k}(w,x)$$
where $\Lambda^{k}$ is the $k$-th exterior product. We suppose that for some $\varepsilon>0$ it holds
\begin{equation*}
\beta(F):=(1-\varepsilon)\int \log \deg(f_w) d \mathbb{P}(w)  - \max_{1\leqslant k\leqslant d -1}\int_{X} C_{k}(w,F) \,d\mathbb{P}(w) >0
\end{equation*}
For potentials $\phi\in\mathbb{L}^1_{\mathbb{P}}(X, C^\alpha(M))$ satisfying (\ref{cond1}) such that for almost $w\in X$ holds $\sup\phi_w-\inf\phi_w<\varepsilon\int \log \deg(f_w)\ d \mathbb{P}(w)$ we obtain the following result.
\begin{corollary}\label{BO} There exists only one equilibrium state associated to $(F|_\theta, \phi)$.
	\end{corollary}
\begin{proof}
Let $\eta \in \mathcal{M}_{\mathbb{P}}(F|\theta)$ be an ergodic equilibrium state for $(F|_{\theta}, \phi)$. Denote by $\lambda_{1}(w,x)\leq \cdots\leq \lambda_{d}(w,x)$ the Lyapunov exponents of $\eta$ at $(w,x)$. We claim that they are bigger than $\beta(F)>0$. If not, by applying the random version of Margulis-Ruelle's inequality (\cite{Pei1}, Theorem 2.4) we have
	\begin{align*}
	h_{\eta}(F|\theta)&\leq \int\sum^{d}_{i=1}\lambda_{i}^{+}(w,x)d\eta(w,x)\\
	&=\int\lambda^{+}_{1}(w,x)d\eta(w,x)+\int\sum_{i\in\{2,...,d\}}\lambda^{+}_{i}(w,x)d\eta(w,x)\\
	&\leq \beta(F)+\int C_{d-1}(w,x)d\eta(w,x)\le \beta(F)+\max_{1\le k\le d-1}\int C_{k}(w,F) \, d\mathbb{P}(w)\\
	&\leq (1-\varepsilon)\int \log \deg(f_w) d \mathbb{P}(w).
	\end{align*}
	Thus for potentials such that $\sup\phi_w-\inf\phi_w<\varepsilon\int \log \deg(f_w)\ d \mathbb{P}(w)$ follows that
	\begin{eqnarray*}
	h_{\eta}(F|\theta) + \int \phi \ d\eta &\leq& (1-\varepsilon)\int \log \deg(f_w) d \mathbb{P}(w) + \int\sup\phi_w  d \eta\\
	&<&\int \log \deg(f_w) d \mathbb{P}(w)+ \int \inf \phi_w  d\eta \leq P_{F|\theta}(\phi).
	\end{eqnarray*}
 which is a contradiction. Therefore the Lyapunov exponents of  $\eta$ are bigger than $\beta(F)$ and so, $\eta$ admits generating partitions with small diameter. Applying the proof of Proposition \ref{propestadodeequilibrio}, we have that $\eta$ is absolutely continuous whith respect to $\nu$ and thus, the uniqueness is proved.

\end{proof}

\section{Equilibrium stability}\label{eq.stability}
Consider $(F_k, \phi_k)$ a sequence in $\mathcal{H}$ converging to $(F, \phi)$. For each $k\in \mathbb{N}$, let $\mu_k$ be the non-uniformly expanding equilibrium state of $(F_k, \phi_k)$. We are going to prove that any accumulation point $\mu$ of the sequence $(\mu_k)$ is the non-uniformly expanding equilibrium state of $(F, \phi).$

For each $k\in \mathbb{N}$ consider $\{\mu_{k, w}\}_{w\in X}$ the disintegration of $\mu_k$. From Theorem~\ref{formalismo}, we know that $\mu_{k, w}=h_{k, w}\nu_{k, w}$ where $h_{k, w}$ and $\nu_{k, w}$ satisfy
$$\mathcal{L}_{k,w}^{\ast}\nu_{k,\theta(w)}=\lambda_{k,w}\nu_{k,w}\quad \mathcal{L}_{k,w}h_{k,w}=\lambda_{k,w} h_{k,\theta(w)}\,\,\mbox{with}\,\,\, \lambda_{k,w}=\nu_{k,\theta(w)}(\mathcal{L}_{k,w}(1)).$$
We point out that for any $\psi\in C^\alpha(M)$ we have $\mathcal L_{k, w}(\psi)$ converging to $\mathcal L_{w}(\psi)$ in $C^0$-norm, see a proof of this in \cite{AlvesRamosSiqueira}.

Let $\lambda_w$, $\nu_w$ and $h_w$ as in Theorem~\ref{formalismo} applied to $(F, \phi).$ The main step in the proof of the equilibrium stability is the following.

\begin{proposition} For almost $w\in X$ we have the convergences 
	$$\lambda_{k, w}\to \lambda_w \quad \nu_{k, w}\stackrel{w*}\longrightarrow \nu_w \quad \mbox{and} \quad h_{k, w}\to h_w, \quad \mbox{when $k$ goes to infinity}.$$
\end{proposition}
\begin{proof}
	Recalling that for each $k\in\mathbb{N}$ we have 
	$$ \deg(f_{k, w})e^{\inf \phi_{k, w}} \leq \lambda_{k, w}\leq \deg(f_{k, w})e^{\sup \phi_{k, w}}$$ 
	then the sequence $(\lambda_{k, w})$ admits some accumulation point  $\bar{\lambda}_w.$
	Moreover, taking subsequences, if necessary, there exist probability measures $\bar{\nu}_w$ and $\bar{\nu}_{\theta(w)}$ such that $\nu_{k, w}\stackrel{w*}\longrightarrow\bar{\nu}_w$ and $\nu_{k, \theta(w)}\stackrel{w*}\longrightarrow\bar{\nu}_{\theta(w)}$. 
	We are going to prove that $\mathcal{L}_{w}^{\ast}\bar{\nu}_{\theta(w)}=\bar{\lambda}_w\bar{\nu}_w.$ 
	
	For any $\psi\in C^\alpha(M)$ we can write 
	$$	\mathcal{L}_{w}^{\ast}\bar{\nu}_{\theta(w)}(\psi)=\bar{\nu}_{\theta(w)}( \mathcal{L}_{w}(\psi))= \bar{\nu}_{\theta(w)}(\lim_{k\to\infty} \mathcal{L}_{k,w}(\psi))=\lim_{k\to\infty}\bar{\nu}_{\theta(w)}\left(\mathcal{L}_{k,w}(\psi)\right).$$
	From the convergence of $\nu_{k, \theta(w)}$ to $\bar{\nu}_{\theta(w)}$ we have
	$$	\lim_{k\to\infty}\bar{\nu}_{\theta(w)}\left(\mathcal{L}_{k,w}(\psi)\right)	=\lim_{k\to\infty}\nu_{k, \theta(w)}\left( \mathcal{L}_{k, w}(\psi)\right) 
	=\lim_{k\to\infty} \mathcal{L}_{k, w}^{\ast}(\nu_{k, \theta(w)})(\psi).
	$$
	Since $\nu_{k, \theta(w)}$ is a reference measure, the last equality can be rewrite as
	$$\lim_{k\to\infty} \mathcal{L}_{k, w}^{\ast}(\nu_{k, \theta(w)})(\psi)= \displaystyle \lim_{k \to \infty} {\lambda_{k, w}} \nu_{k, w}(\psi)=  \bar{\lambda}_w \bar{\nu}_w (\psi). 
	$$
	Thus $\mathcal{L}_{w}^{\ast}\bar{\nu}_{\theta(w)}(\psi)=\bar{\lambda}_w \bar{\nu}_w (\psi)$ for any $\psi\in C^\alpha(M)$. Because $C^\alpha(M)$ is dense in $C^0(M)$ we conclude that $\mathcal{L}_{w}^{\ast}\bar{\nu}_{\theta(w)}=\bar{\lambda}_w\bar{\nu}_w.$ 
	
	Now we are going to verify that $\bar{\lambda}_w=\lambda_w$. Therefore, from the uniqueness given by Theorem~\ref{formalismo} follows that $\bar{\nu}_w=\nu_w.$
	
		Given $\varepsilon>0$ small and $n\in \mathbb{N}$, consider $F_n$ a $(w, n, \varepsilon)$-separated set. Let $\mathcal{U}$ be the open cover of $M$ defined by $\mathcal{U}:=\{\displaystyle\cap_{j=0}^{n-1}f^{-j}_{w}(B(f_{w}^{j}(x), \varepsilon))\, ;\, x\in F_n\}$. Because $(\mathcal{L}^{n}_{w})^{\ast}\bar{\nu}_{\theta^{n}(w)}=\bar{\lambda}^{n}_{w}\bar{\nu}_{w}$ it follows that
	\begin{eqnarray*}
		1=\bar{\nu}_{w}(M)&=&\int(\bar{\lambda}^{n}_{w})^{-1}\mathcal{L}^{n}_{w}(1)\ d\bar{\nu}_{\theta^{n}(w)}\\
		&\leq& (\bar{\lambda}^{n}_{w})^{-1}\sum_{U\subset \mathcal{U}}\int_{U}e^{S_{n}\phi_w(z)}\ d\bar{\nu}_{\theta^{n}(w)}\\
		&\leq&(\bar{\lambda}^{n}_{w})^{-1}\sum_{x\in F_n}e^{S_{n}\phi_w(x)}\int_{U}e^{(S_{n}\phi_w(z)-S_{n}\phi_w(x))}\ d\bar{\nu}_{\theta^{n}(w)}\\
		&\leq& (\bar{\lambda}^{n}_{w})^{-1}\displaystyle\sum_{x\in F_n}e^{S_{n}\phi_w(x)}e^{\sum_{j=0}^{n-1}|\phi_{\theta^{j}(w)}|_{\alpha}\varepsilon}.
	\end{eqnarray*}
	Thus $\sum_{j=0}^{n-1}(\log \bar{\lambda}_{\theta^{j}(w)}-|\phi_{\theta^{j}(w)}|_{\alpha}\varepsilon)\leq \log P_{F|_\theta}(w, n, \varepsilon).$ As $\mathbb{P}$ is ergodic we obtain
	\begin{eqnarray*}	\int\!\log\bar{\lambda}_w \ d\mathbb{P}(w)-\varepsilon\int|\phi_{w}|_{\alpha}d\mathbb{P}(w)&=& \lim_{n\to \infty}\dfrac{1}{n}\sum_{j=0}^{n-1}\log \bar{\lambda}_{\theta^{j}(w)}-\varepsilon\dfrac{1}{n}\sum_{j=0}^{n-1}|\phi_{\theta^{j}(w)}|_{\alpha}\\
		&\leq&\limsup_{n\to \infty}\dfrac{1}{n}\int \log P_{F|_\theta}(w, n, \varepsilon)\ d\mathbb{P}(w) 
	\end{eqnarray*}		
	for every $\varepsilon>0$ small. Hence $\int\!\log\bar{\lambda}_w \ d\mathbb{P}(w)\leq P_{F}(\phi)$. On the other hand, since $\bar{\nu}_w$ is a reference mesure, it satisfies a Gibbs property on hyperbolic times (Proposition~\ref{conforme}). Thus we can apply the proof of Proposition~\ref{pressao} to obtain $P_{F}(\phi)\leq\int\log\bar{\lambda}_w\ d\mathbb{P}(w)$. Recalling that $P_{F}(\phi)=\int\log\lambda_w\ d\mathbb{P}(w)$ we have 	$$\int\log\bar{\lambda}_w\ d\mathbb{P}(w)\leq P_{F}(\phi)=\int\log{\lambda}_w\ d\mathbb{P}(w)\leq\int\log\bar{\lambda}_w\ d\mathbb{P}(w).$$
	Since this is constant $\mathbb{P}$-almost every $w\in X$, we have proved that $\bar{\lambda}_w=\lambda_w$.


	To finish the proposition it remains to prove the convergence $h_{k,w}\to h_{w}$. Since $(F_k, \phi_k) \in \mathcal{H}$ we can assume that the transfer operator $\mathcal{L}_{k, w}$ preserves the same cone $\mathcal{C}^{\hat{k}}_{\delta}$ for $\hat{k}$ large enough . Then, recalling the proof of Proposition~\ref{h} we have each $h_{k, w}\in \mathcal{C}^{\hat{k}}_{\delta}$ with $\int h_{k,w}\ d\nu_{k, w}=1$. Moreover, it satisfies 
	$$|h_{k,w}(x)-h_{k,w}(y)|\leq Cm\ d(x, y)^{\alpha}\quad \mbox{and} \quad |h_{k,w}(y)|\leq \sup h_{k,w}\leq R.$$
	Therefore $(h_{k,w})$ is a equicontinuous and uniformly bounded sequence. From Ascoli-Arzel\`a's Theorem there exists $\bar{h}_w$ some accumulation point. Notice that $\bar{h}_w$ is H\"older continuous and $\int \bar{h}_w\ d\nu_w=1$ because $\nu_{k, w}\stackrel{w*}\longrightarrow \nu_w$. Moreover, $\bar{h}_w$ satisfies
	\begin{eqnarray*}
		\mathcal{L}_{w}(\bar{h}_{w})=\mathcal{L}_{w}(\lim_{k\to\infty}h_{k, w})=\lim_{k\to\infty}\mathcal{L}_{w}(h_{k, w})&=&\lim_{k\to\infty}\mathcal{L}_{k,w}(h_{k, w})\\
		&=&\lim_{k\to\infty}\lambda_{k, w}h_{k,\theta(w)}=\lambda_w \bar{h}_{\theta(w)}
	\end{eqnarray*}
	By the uniqueness of Theorem~\ref{formalismo} we obtain $\bar{h}_w=h_w$ almost every where.
\end{proof}

From the previous result we obtain for almost $w\in X$ that the sequence $(\mu_{k, w})$ converges to $\mu_w$ defined by $\mu_w=h_w\nu_w.$ Therefore, the sequence $(\mu_k)$ converges to the probability measure $\mu$ whose disintegration is $\{\mu_w\}_{w\in X}$. As in Section~\ref{ee}, we have that $\mu$ is the non-uniformly expanding equilibrium state of $(F|_\theta, \phi).$ Moreover, since in the family $\mathcal{H}$ it holds that $P_{F|_\theta}(\phi)=\int\log\lambda_w\ d\mathbb{P}(w)$ we obtain
$$ P_{F|_\theta}(\phi)=\int\log\lambda_w\ d\mathbb{P}(w)=\lim_{k\to\infty}\int\log\lambda_{k,w}\ d\mathbb{P}(w)=\lim_{k\to\infty}P_{F_k}(\phi_k)$$
which proves that the random topological pressure varies continuously in the family. This finishes the proof of Theorem~\ref{estabilidade}.

\section{Applications}

	 In this section we present some classes of systems which satisfy our results. We start by describing a robust class of local diffeomorphisms which contains an open set of non-uniformly expanding maps that are not uniformly expanding. This class was studied in the deterministic case by several authors~\cite{AlvesAraujo, Alves3, Varandas1, VarandasViana}. The first example is an one dimensional version of this class.

\begin{example}\normalfont{Let $g:\mathbb{S}^1\to \mathbb{S}^1$ be a $C^1 $-local diffeomorphisms defined on the unit circle. Fix $\delta>0$ small, $\sigma<1$ and consider a covering $\mathcal Q$ of $\mathbb{S}^1$ by injectivity domains of $f$ and a region $\textsl{A}\subset \mathbb{S}^1$ covered by $q$ elements of $\mathcal Q$ with $q<\deg(g)$ such that
		\begin{enumerate}
			\item[(H1)]   $\|Dg^{-1}(x)\|\leq 1+\delta$, for every $x\in\textsl{A}$;
			\item[(H2)] $\|Dg^{-1}(x)\|\leq \sigma $, for every $x\in M\setminus\textsl{A}$;
		\end{enumerate}
Denote by $\mathcal{F}$ the class of $C^1 $-local diffeomorphisms satisfying conditions (I)-(II). We also assume that every $g\in \mathcal{F}$ is \emph{topologically exact} and its degree $\deg g$ is constant. Notice that $\mathcal{F}$ contains expanding maps, perturbations of expanding maps and intermittent maps. 

Let $\theta: \mathbb{S}^1\to \mathbb{S}^1$ be any invertible function preserving an ergodic measure $\mathbb{P}$ on $\mathbb{S}^1$. Thus any random dynamical system $f=(f_w)_{w}$ generated by maps $f_w\in \mathcal{F}$ satisfies the hypotheses of our theorems. For potentials $\phi \in \mathbb{L}_{\mathbb{P}}^{1}(\mathbb{S}^1, C^{\alpha}(\mathbb{S}^1))$ satisfying (\ref{cond1}) we can apply our results to obtain the thermodynamical formalism in this class and the existence of only one equilibrium state on the set of non-uniformly expanding measures. 

Moreover, if the potential also satisfies the condition $\sup\phi< P_\phi(f)$ then the equilibrium state is unique in the class of ergodic measures. Indeed, using the random versions of Oseledt's theorem and Ruelle's inequality (see \cite{Pei1}), for the equilibrium state $\mu$ the Lyapunov exponent $\lambda(\mu)$ satisfies
\begin{eqnarray*} \lambda(\mu)\geq h_\mu(f)=P_\phi(f)-\int \phi\ d\mu&\geq& h_{top}(f)+\inf\phi-\sup \phi\\
	&\geq& h_{top}(f)-(\sup \phi-\inf\phi)\geq \log q >0.
	\end{eqnarray*}
Therefore, $\lambda(\mu)$ is positive and bounded away from zero. In dimension one this implies that the equilibrium state is non-uniformly expanding.

}
	
\end{example}

The second example is a generalization of the previous one in higher dimension. The existence of equilibrium state for random transformations given by maps in this setting was considered by Arbieto, Matheus and Oliveira \cite{Oliveira2}. 

\begin{example}\normalfont{
Let $M^l$ be a compact $l$-dimensional Riemannian manifold and $\mathcal{D}$ the space of $C^2$ local diffeomorphisms on $M$. Let $(\Omega, T, \mathbb{P})$ be a measure preserving system where $\mathbb{P}$ is ergodic. Define the skew-product by
$$
\begin{array}{cccc}
F\ : & \! \Omega\times M & \! \longrightarrow
& \!\Omega\times M \\
& \! (w,x) & \! \longmapsto	 & \! (T(w),f(w)x)
\end{array}
$$ 
where the maps $f(w)\in \mathcal{D}$ varies continuously on $w\in \Omega$. Fixing positive constants $\delta_0, \delta_1$ small and $p, q \in \mathbb{N}$, satisfying for every $f(w)\in \mathcal{D}$ the following properties:
\begin{enumerate}
    \item [(H1)] There exists a covering $B_1,...,B_p,...,B_{p+q}$ of $M$ by injectivity domains s.t.
    \begin{itemize} 
    \item $\|Df(x)^{-1}\|\leq (1+\delta_1)^{-1}$ for every $x\in B_1 \cup \cdots \cup B_p$.
    \item  $\|Df(x)^{-1}\|\leq (1+\delta_0)$ for every $x\in M$.
    \end{itemize}
 \item [(H2)] $f$ is everywhere volume expanding: $|\det Df(x)| \geq \sigma_1$ with $\sigma_1 > q$.
 \item [(H3)] There exists $A_0$ s.t. $|\log \|f\|_{C^2} | \le A_0$ for any $f\in \mathcal{F} \subset \mathcal{D}$.
\end{enumerate}

 Adding other technical hypotheses, the authors in \cite{Oliveira2} have showed the existence of equilibrium states for potentials with small variation. Moreover, they proved that theses measures are non-uniformly expanding. Now, for potentials satisfying condition~(\ref{cond1}), we can apply our results to obtain the thermodynamical formalism and the uniqueness of equilibrium state for this class.
Let $\mathcal{S}$ be the set of skew-products generated by maps of $\mathcal{D}$ where $T:\Omega\to\Omega$ is fix:
$$F:X\times M \to X\times M \ \ ; \ \ F(w, x)=(T(w),f_{w}(x))$$
Define the family 
$$\mathcal{H}=\left\{(F, \phi)\in \mathcal{S}\times \mathbb{L}^{1}_{\mathbb{P}}(X,C^\alpha(M)) \, ;\, \phi\,\, \mbox{satisfies (\ref{cond1})}\right\}.$$

Notice that $\mathcal{H}$ satisfies the hypothesis of Theorem~\ref{estabilidade}. Thus, the equilibrium state and the random topological pressure vary continuously within this family.}
\end{example}

Next we present an application of our Corollary~\ref{BO}. This example appears in \cite{Bilbao} in the context of maximizing entropy measures. Here we prove uniqueness of equilibrium states for potentials with small variation. 
\begin{example}
	\label{example1}\normalfont{
	Let $f_{0},f_{1}: M\rightarrow M$ be $C^1$ local diffeomorphisms of a compact and connected manifold $M$ satisfying our conditions (I)-(III). For $1\leq k < \dim M=d$ suppose that $\log \lVert \Lambda^{k} Df_1\rVert < \log \deg f_1$ and consider $$C_{k}(w,x)=\limsup_{n\rightarrow +\infty} \frac{1}{n}\log \|\Lambda^{k}Df_{w}^{n}(x)\|\quad \mbox{and} \quad	C_{k}(w)=\max_{x \in M}C_{k}(w,x).$$
	Let $\mathbb{P}_\alpha$ be the Bernoulli measure on the sequence space $X=\{0,1\}^\mathbb{Z}$ such that $\mathbb{P}_\alpha([1])= \alpha$.
	In  \cite{Bilbao} was proved the existence of $\alpha\in (0, 1)$ close to $1$ such that
		\begin{eqnarray*}\int\lim_{n\to \infty} \frac{1}{n}\log \lVert \Lambda^{k} Df^{n}_{w}(x)\rVert\ d\mathbb{P}_{\alpha}(w)&<& \alpha \log \deg(f_1) + (1-\alpha)\log \deg (f_0) \\
		&=& \int \log \deg (f_w)\  d\mathbb{P}_{\alpha}(w).
			\end{eqnarray*}
for every $x\in M.$ Therefore, for some $\varepsilon>0$ we have 
$$(1-\varepsilon)\int \log \deg(f_w)\ d\mathbb{P}_{\alpha}(w)  - \max_{1\leqslant k\leqslant d -1}\int_{X} C_{k}(w) \,d\mathbb{P}_{\alpha}(w) >0$$
which means that the hypothesis of Corollary~\ref{BO} was verified. Thus, for potentials $\phi\in \mathbb{L}^1_{\mathbb{P}}(X, C^{\alpha}(M))$ satisfying~(\ref{cond1}) such that $\sup\phi_w-\inf\phi_w<\varepsilon\int \log \deg(f_w) d \mathbb{P}_{\alpha}$ we conclude uniqueness of equilibrium states.
}
\end{example}


%
	%

\vspace{0.5cm}
\section*{Acknowledgements}
We would like to thank P. Varandas for useful suggestions and encouragement. RB thanks to K. Oliveira for many conversations and to IM-UFAL for the hospita-lity. VR also thanks to FAPEMA-Brazil for its financial support.

\end{document}